\date{21 September 2022}

\documentclass[11pt]{article}

\usepackage{graphicx,amsfonts,psfrag}
\usepackage{amsmath,amsthm,amssymb}
\usepackage{xcolor}

\newtheorem{theorem}{Theorem}
\newtheorem{lemma}[theorem]{Lemma}

\newtheorem{corollary}[theorem]{Corollary}

\newcommand{\cA}{{\mathcal A}}

\newcommand{\cAd}{{\mathcal A}^{\delta\geq 2}}
\newcommand{\cB}{{\mathcal B}}
\newcommand{\cC}{{\mathcal C}}

\newcommand{\cCd}{{\mathcal C}^{\delta\geq 2}}
\newcommand{\tC}{\tilde{\cC}}
\newcommand{\cD}{{\mathcal D}}

\newcommand{\cE}{{\mathcal E}}
\newcommand{\cEr}{{\mathcal E}^{\bullet}}
\newcommand{\Er}{{E}^{\bullet}}
\newcommand{\cF}{{\mathcal F}}

\newcommand{\cFor}{{\mathcal F}}

\newcommand{\cG}{{\mathcal G}}
\newcommand{\cGd}{{\mathcal G}^{\delta\geq 2}}
\newcommand{\tG}{\tilde{\cG}}
\newcommand{\cH}{{\mathcal H}}
\newcommand{\cHb}{{\mathcal H}^{\bullet}}

\newcommand{\cP}{{\mathcal P}}

\newcommand{\cS}{{\mathcal S}}
\newcommand{\cSr}{{\mathcal S}^{\bullet}}
\newcommand{\cT}{{\mathcal T}}
\newcommand{\cTr}{{\mathcal T}^{\bullet}}
\newcommand{\Tr}{{T}^{\bullet}}

\newcommand{\N}{\mathbb N}
\newcommand{\E}{\mathbb E}
\newcommand{\pr}{\mathbb P}
\newcommand{\whp}{whp }

\def\Po{\mbox{\rm Po}}
\newcommand{\inu}{\in_{u}}

\newcommand{\aut}{\mbox{\small aut}}
\newcommand{\core}{{\rm core}}
\newcommand{\Core}{{\rm Core}}
\newcommand{\icore}{{\rm icore}}
\newcommand{\iCore}{{\rm iCore}}

\newcommand{\bigc}{\rm big}
\newcommand{\Bigc}{\rm Big}
\newcommand{\Frag}{\mbox{{\rm Frag}}}
\newcommand{\frag}{{\mbox{\rm frag}}}

\newcommand{\Hb}{H^{\bullet}}
\newcommand{\pend}{{\rm pend}}
\newcommand{\orb}{{\rm orb}}

\newcommand{\m}[1]{\marginpar{\tiny{#1}}}
\newcommand{\eps}{\epsilon}
\newcommand{\cc}[1]{{\color{red}#1}}

\newcommand{\gap}{{\rm gap}}

  {\hspace*{\fill}$\Box$}

\begin{document}

\title{Random graphs from structured classes}

\author{Colin McDiarmid \\
\small  Department of Statistics \\
\small Oxford University, UK\\
\small \tt cmcd@stats.ox.ac.uk}

\maketitle

\begin{abstract}
Given a class $\cG$ of graphs, let $\cG_n$ denote the set of graphs in $\cG$ on vertex set $[n]$.  For certain classes $\cG$, we are interested in the asymptotic behaviour of a random graph $R_n$ sampled uniformly from $\cG_n$.
Call $\cG$ \emph{smooth} if $ n |\cG_{n-1}| / |\cG_n|$ tends to a limit as $n \to \infty$.  Showing that a graph class is smooth is a key step in an approach to investigating properties of $R_n$, in particular the asymptotic probability that $R_n$ is connected, and more generally the asymptotic behaviour of the fragment of $R_n$ outside the largest component.

The composition method of Bender, Canfield and Richmond~\cite{bcr08} shows that the class of graphs embeddable in a given surface is smooth; and similarly we have smoothness for any minor-closed class of graphs with 2-connected excluded minors~\cite{cmcd2009}.
Here we develop the approach further, and give results encompassing both these cases and much more.
%
We see that, under quite general conditions, our graph classes are smooth and we can describe for example the limiting distribution of the fragment of $R_n$ and the size of the core; and we obtain similar results for the graphs in the class with minimum degree at least 2.
%
%
\end{abstract}


\section{Introduction and statement of results}
\label{sec.intro}


Given a class $\cG$ of graphs (closed under isomorphism), let $\cG_n$ denote the set of graphs in $\cG$ on vertex set $[n]$ for non-negative integers $n$.  The notation $R_n \in_u \cG$ indicates that the random graph $R_n$ is uniformly distributed over $\cG_n$ (with the implicit assumption that $\cG_n \neq \emptyset$).
In order to deduce results about $R_n$, we often need to know that the numbers $|\cG_n|$ do not behave too erratically as $n$ varies.  Let $G(x)$ denote the corresponding (exponential) generating function $\sum_{n \geq 0} |\cG_n|x^n/n!$, and let $\rho_{\cG}$ (or sometimes $\rho(\cG)$) denote the radius of convergence of $G(x)$, 
$0 \leq \rho_{\cG} \leq \infty$.
If $0 < \gamma < \infty$ and
\[  (|{\cG}_n| / n!)^{1/n} \to \gamma \;\; \mbox{ as } n \to \infty, \]
we say that $\cG$ has \emph{growth constant} $\gamma$.  In this case $\gamma$ must be $\rho_{\cG}^{-1}$.

%

Showing that $\cG$ has a growth constant is a first step in proving results about $R_n \inu \cG$: a next step is to show smoothness. Suppose that $0< \rho_{\cG}< \infty$.  Let $r_n = n |\cG_{n-1}|/|\cG_n|$. Observe that, if $\cG$ is closed under adding and removing isolated vertices, then $r_n$ is the expected number of isolated vertices in $R_n$.  It is easy to see that
\begin{equation} \label{eqn.rho}
  \liminf_{n \to \infty} r_n \leq \rho_{\cG} \leq
  \limsup_{n \to \infty} r_n.
\end{equation}
Thus if $r_n$ tends to a limit as $n \to \infty$ it must be $\rho_{\cG}$.  We call $\cG$ \emph{smooth} if $r_n \to \rho_{\cG}$ as $n \to \infty$. It is easy to see that in this case $\cG$ must have a growth constant.

The classes for which we know an asymptotic counting formula are typically smooth, for example the class $\cF$ of forests~\cite{renyi59}, the classes of outerplanar graphs and of series-parallel graphs~\cite{bgkn05,bgkn07}, 
the class $\cP$ of planar graphs~\cite{gn09a,gn09b}, and indeed the class $\cE^S$ 
of graphs embeddable in a given surface $S$ \cite{bg09,cfgmn11}. In particular, for forests we have $|\cFor_n| \sim e^{\frac12} n^{n-2}$ (see~\cite{renyi59})
and so
\[ \frac{n |\cF_{n-1}|}{|\cFor_n|} \sim \frac{n(n-1)^{n-3}}{n^{n-2}} \sim e^{-1} = \rho_{\cF}. \]
(Of course, we would need to make adjustments for example if $\cG$ is the class of cubic planar graphs, see~\cite{blkm05}, where $|\cG_n|=0$ for odd integers $n$.)

If we do not know an asymptotic counting formula, in order to prove smoothness we may be able to use the composition method of Bender, Canfield and Richmond~\cite{bcr08}, see section~\ref{sec.bcr} below.
In this way we can show that the class of graphs embeddable in a given surface is smooth~\cite{bcr08}, 
and that any minor-closed graph class 
which has 2-connected excluded minors is smooth~\cite{cmcd2009}.
The proofs of our main results here, Theorems~\ref{thm.smooth-likeGS} and~\ref{thm.smooth-Frag},
are based on this approach with some further development.
We focus on structured classes of graphs satisfying some natural conditions, so that our results apply to surface classes $\cE^S$ and addable minor-closed classes and much more.

Related results are obtained in~\cite{cmcdPAC2021} by a different proof approach, which requires the graph class to have a growth constant but not necessarily to be smooth.

\subsection{Preliminaries}
We now introduce some properties which we shall need for a graph class~$\cG$. 
The \emph{null graph} (or empty graph) has no vertices or edges. By convention, a class $\cG$ of graphs contains the null graph unless it is explicitly a class of connected graphs, when it does not contain the null graph.  The class $\cG$ is \emph{proper} if it contains a graph other than the null graph and it does not contain all graphs.  We may sometimes assume implicitly that this holds. 
We call $\cG$ \emph{decomposable} when a graph is in $\cG$ if and only if each component is. Observe that the class of planar graphs is decomposable but the class of graphs embeddable in the torus is not.  (For example, the complete graph $K_5$ embeds in the torus but two disjoint copies of $K_5$ do not.) Indeed the surface class $\cE^S$ is decomposable only if $S$ is the sphere. If $\cG$ is decomposable, and $\cC$ is the class of connected graphs in $\cG$ with generating function $C(x)$, then the \emph{exponential formula} says that $G(x)=e^{C(x)}$.

A graph $H$ is a \emph{minor} of a graph $G$ if we may obtain a copy of $H$ by starting with a subgraph of $G$ and contracting edges (we replace any multiple edges by a single edge, and delete loops).  A class of graphs is \emph{minor-closed} if whenever $G \in \cG$ and $H$ is a minor of $G$ we have $H \in \cG$.   If $\cG$ is minor-closed the minor-minimal graphs not in $\cG$ are the \emph{excluded minors} for $\cG$.  The Robertson-Seymour Theorem says that for every minor-closed class the set of excluded minors is finite, see for example~\cite{Diestelbook}. It is easy to see that a minor-closed class is decomposable if and only if each excluded minor is connected.

We call a graph class $\cG$ \emph{bridge-addable} if whenever $G \in \cG$  and $u$ and $v$ are in different components of $G$ then $G+uv \in \cG$. Here $G+uv$ denotes the graph obtained by adding the edge $uv$ to $G$.
(This property was called `weakly addable' in~\cite{msw06}.)
By Theorem 2.2 of~\cite{msw05} or~\cite{msw06}, if $\cG$ is bridge-addable and $R_n \inu \cG$ then
\begin{equation} \label{eqn.ba}
\pr(R_n \mbox{ is connected}) \geq 1/e\,.
\end{equation}
%
Following~\cite{msw06}, we say that the class $\cG$ is \emph{addable}
if it is both decomposable and bridge-addable.  For example, the class $\cP$ of planar graphs is addable.  Recall that $\cP$ is minor-closed, with excluded minors $K_5$ and $K_{3,3}$. In general, a minor-closed class $\cG$ is addable if and only if each excluded minor is 2-connected, as noted in~\cite{cmcd2009}.
The surface class $\cE^S$ is always bridge-addable, but as we already noted it is not decomposable (and so not addable) except in the planar case.

Given a graph $G$, the \emph{core} (or 2-{\em core}) $\Core(G)$ of $G$ is the null graph if $G$ is a forest, and otherwise it is the unique maximal subgraph of $G$ with minimum degree at least~2. Thus $\Core(G)$ is the graph obtained by repeatedly trimming off leaves (in any order) until none remain, and then deleting any isolated vertices.  The number of vertices in $\Core(G)$ is denoted by $\core(G)$.
Given a class $\cG$ of graphs we let $\cG^{\delta \geq 2}$ denote the class of graphs in~$\cG$ with minimum degree at least two.

We say that a class $\cG$ of graphs 
is \emph{trimmable} if for any graph $G$
\begin{equation} \label{cond.trim}
G \in \cG \; \mbox{ if and only if } \: \Core(G) \in \cG .
\end{equation}
Observe that a trimmable graph class which contains the null graph must contain all forests.
We may see that $\cG$ is trimmable
precisely when, for every graph $G$ with a leaf $v$, $G \in \cG$ if and only if $G\!-\!v \in \cG$. (Here $G\!-\!v$ denotes $G$ with $v$ and the incident edge removed.) 
Also, if $\cG$ is hereditary (closed under forming induced subgraphs), closed under adding an isolated vertex and bridge-addable, then $\cG$ is trimmable. 

Clearly each surface class $\cE^S$ is trimmable.
Consider a minor-closed class $\cG$ of graphs.
It is not hard to see that $\cG$ is trimmable if and only if each excluded minor has minimum degree at least two.
Also, if each excluded minor has minimum degree at least one (so $\cG$ is closed under adding an isolated vertex) and $\cG$ is bridge-addable then $\cG$ is trimmable.
[The class of graphs with no minor the `bowtie graph' (obtained from two disjoint triangles by identifying a vertex in one triangle with a vertex from the other triangle) is an example of a graph class which is trimmable but not bridge-addable.]  See also the paragraph below where we define when a graph is `free for $\cG$' (before Theorem 2).
 
 The \emph{fragment} $\Frag(G)$ of a graph $G$ is the graph remaining (perhaps the null graph) when we discard the component $\Bigc(G)$ with the most vertices (breaking ties lexicographically or in some other way).
If $G$ has vertex set $[n]$ then $\Frag(G)$ has vertex set some subset of $[n]$. 
it is natural to consider the fragment as an unlabelled graph, and we shall do so. 

\smallskip


Finally in this subsection let us recall the Boltzmann Poisson random graph, see~\cite{cmcd2009}. 
Let the class $\cG$ of graphs be decomposable, and let $\tG$ denote the set of unlabelled graphs in $\cG$.  (Recall that by convention the null graph $\emptyset$ is in $\cG$ and in $\tG$.) Fix $\rho >0$ such that $G(\rho)$ is finite; and let
\[  \mu(H) = \frac{\rho^{v(H)}}{\aut(H)} \; \mbox{ for each } H \in \tG \]
with $\mu(\emptyset)=1$.
Here $v(H)$ is the number of vertices in $H$, and $\aut(H)$ is the size of the automorphism group.
Routine manipulations (see for example~\cite{cmcd2009}) show that
\[   G(\rho)= \sum_{H \in \tG} \mu(H). \]
%
The \emph{Boltzmann Poisson random graph} $R=BP(\cG,\rho)$ takes values in $\tG$, with
\begin{equation} \label{eqn.BP}
  \pr(R=H) = \frac{\mu(H)}{G(\rho)}  \;\; \mbox{ for each } H \in \tG.
\end{equation}
Let $\cC$ denote the class of connected graphs in $\cG$.  (Recall that by convention $\emptyset$ is not in $\cC$.)  For each $H \in \tC$ let $\kappa(G,H)$ denote the number of components of $G$ isomorphic to $H$. Then (see \cite{cmcd2009}) the random variables $\kappa(R,H)$ for $H \in \tC$ are independent, with $\kappa(R,H) \sim \Po(\mu(H))$. In particular, 
$\pr(R=\emptyset) = e^{-C(\rho)} = 1/G(\rho)$.

\subsection{Statement of main results} 

We first present Theorem~\ref{thm.smoothold}, in order to bring together previously published results in this area.  It is taken from Theorems 1.2 and 1.5 of~\cite{cmcd2009}, and Theorems 2.2 and 2.4 of~\cite{cmcd2013}.  (The latter paper concerns more general weighted distributions: we stick to the uniform distribution here.) Recall that a proper addable minor-closed class of graphs is trimmable, and contains the class $\cF$ of forests.



\begin{theorem} \label{thm.smoothold}
Let $\cG$ and $\cA$ be graph classes with $\cG \supseteq \cA$; and suppose that either (i) $\cG$ is a proper addable minor-closed class and $\cA$ is $\cG$, or (ii) $\cG$ is the class $\cE^S$ of graphs embeddable in a given surface $S$ and $\cA$ is the class $\cP$ of planar graphs.  Let $\cC$ be the class of connected graphs in $\cG$, and let $\rho_0 =\rho_\cA$. 
Then the following statements (a), (b) hold.
\begin{description}
\item{(a)} Both $\cG$ and $\cC$ are smooth, with radius of convergence $\rho_0$, $\, 0<\rho_0 \leq 1/e$\,. 

\item{(b)} 
$\, 0< A(\rho_0)< \infty$, so $BP(\cA,\rho_0)$ is well defined; and for $R_n \in_{u} \cG$, the fragment of $R_n$ converges in distribution to $BP(\cA,\rho_0)$.  In particular, the probability that $R_n$ is connected tends to $e^{-C(\rho_0)}$ 
as $n \to \infty$.
\end{description}

\noindent
Now add the condition in case (i) that $\cA$ contains a graph with a cycle.
Then the following statements (c), (d) hold.
\begin{description}
\item{(c)} 
$0<\rho_0< 1/e$, and
 $\rho_2 := \rho(\cAd)$ is the unique root $x$ with $0<x< 1$ to $x e^{-x} = \rho_0\,$ (so $\rho_2>\rho_0$); and both $\cGd$ and $\cCd$ have growth constant $\rho_2^{-1}$.
 \item{(d)}
Given $\eps>0$, for $R_n \in_u \cG$ or $R_n \in_u \cC$ we have
\begin{equation} \label{eqn.expcorebd}
   \pr(|\,\core(R_n) -(1-\rho_2) n\,| > \eps n) = e^{-\Omega(n)}.
\end{equation}

\end{description}
\end{theorem}


We shall give two new theorems, which between them extend Theorem~\ref{thm.smoothold} in two directions.
They provide a generalisation of the relationship between the class $\cE^S$ 
of graphs embeddable in a given surface $S$ and the subclass $\cP$ of planar graphs.
Note that the only proper addable minor-closed class $\cA$ of graphs with $\rho_\cA \geq 1/e$ is the class $\cF$ of forests, with $\rho_\cF=1/e$.
(For $\cA$ must contain $\cF$, and if $\cA \neq \cF$ then it contains a cycle $H$, and by the Pendant Appearances Theorem (see Theorem 4.1 in~\cite{msw05})  we have $\rho_\cA < \rho_\cF=1/e$.)
When we extend Theorem~\ref{thm.smoothold} we shall exclude this case.

Observe that $\cE^S$ is bridge-addable and trimmable, $\cE^S \supseteq \cP$ and $\rho(\cE^S)=\rho_{\cP}$   (\cite{cmcd2008}, see also \cite{bg09,cfgmn11}).  The original application of the composition method was to show that both $\cE^S$ and the subclass of connected graphs in $\cE^S$ are smooth classes~\cite{bcr08}.
Given a class $\cG$ of graphs, call a connected graph $H$ \emph{free} for $\cG$ when, for each graph $G$ disjoint from $H$, the following statements (a), (b) and (c) are equivalent: (a) $G \in \cG$; (b) $G \cup H \in \cG$; (c) $G' \in \cG$, where $G'$ is formed from $G \cup H$ by adding an edge (bridge) between a vertex in $H$ and a vertex in $G$.
Call a general graph \emph{free} for $\cG$ when each component is free for $\cG$.
(We also say that the null graph is free for any class $\cG$.)
Observe that $\cG$ is trimmable if and only if the single vertex graph $K_1$ is free for~$\cG$.  Also, given a trimmable class $\cG$, a connected graph is free for $\cG$ if and only if its core is free for $\cG$.


In our first new theorem, Theorem~\ref{thm.smooth-likeGS}, the first three parts correspond to parts (a), (c) and (d) of Theorem~\ref{thm.smoothold}.  We defer consideration of the fragment of $R_n$ (corresponding to part (b) of Theorem~\ref{thm.smoothold}) until Theorem~\ref{thm.smooth-Frag}.

%
%
\begin{theorem} \label{thm.smooth-likeGS}
Let $\cG$ be a trimmable class of graphs
 with $0<\rho_\cG < \infty$, 
and let $\cC$ be the class of connected graphs in~$\cG$.
Suppose that there is an addable subclass $\cA$ of $\cG$ such that $\cA$ contains a graph which has a cycle, and $\rho_{\cA}=\rho_{\cG} := \rho_0$. 
Then the following statements hold.
\begin{description}

\item{(a)}
Both $\cG$ and $\cC$ are smooth, with growth constant 
$\rho_0^{-1}$.

\item{(b)} 
$0<\rho_0 < 1/e$;
 $\rho_2 := \rho(\cAd)$ is the unique root $x$ with $0<x< 1$ to $x e^{-x} = \rho_0$; and both $\cGd$ and $\cCd$ have growth constant $\rho_2^{-1}$.

\item{(c)}
Given $\eps>0$, for $R_n \in_u \cG$ or for $R_n \in_u \cC$ we have
\[ \pr(|\,\core(R_n)- (1-\rho_2) n\,| > \eps n) = e^{-\Omega(n)}. \]


\item{(d)}  If some cycle
 is free for $\cG$ 
then both $\cG^{\delta \geq 2}$ and $\cC^{\delta \geq 2}$ are smooth.

\end{description}
\end{theorem}

Before we give our second new theorem we give some examples for Theorem~\ref{thm.smooth-likeGS}.
Two natural examples come from Theorem~\ref{thm.smoothold}: they
are (i) when $\cG$ is minor-closed and addable, and $\cA=\cG$, 
and (ii) when $\cG$ is a surface class $\cE^S$ and $\cA=\cP$.  
In both cases every cycle in $\cA$ is free for $\cG$. From these examples we may essentially read off parts (a), (c) and (d) of Theorem~\ref{thm.smoothold}.  
Also, 
 by part~(d) of Theorem~\ref{thm.smooth-likeGS}, both $\cG^{\delta \geq 2}$ and $\cC^{\delta \geq 2}$ are smooth: for the case when $\cG$ is an addable minor-closed class, this result was first presented by the author in an invited lecture at RANDOM 2014 in Barcelona.

In the above examples the classes $\cG$ and $\cA$ are minor-closed, but that is not necessarily the case in our further examples.
We could for example let $\cA$ be the class of bipartite planar graphs, 
or planar graphs of girth $t$ (for some integer $t \geq 3$); and let $\cG$ be the corresponding class of graphs embeddable in some given surface~$S$.
Also, we do not need to consider only a fixed surface.
For example, let $g=g(n) \geq 0$ be non-decreasing,
 and let $\cH\cE^g$ be the class of graphs $G$ such that every induced subgraph $G'$ of $G$ can be embedded in a surface of Euler genus at most $g(n')$ where $n'=v(G')$. (The $\cH\cE$ here is for `hereditarily embeddable'.) Observe that $\cH\cE^g$ is trimmable.  It is shown in~\cite{cmcd-ss2021} that if $g(n)=o(n/\log^3n)$ then $\rho(\cH\cE^g) = \rho_{\cP}$.  Thus for such a function $g$, in Theorem~\ref{thm.smooth-likeGS} we could take $\cG=\cH\cE^g$ and $\cA=\cP$.

Finally we give one rather different example for
Theorem~\ref{thm.smooth-likeGS}.
Fix a positive integer $k$, and let $\cG$ be the class of graphs with at most $k$ edge-disjoint subgraphs contractible to a diamond ($K_4$ less an edge).
Clearly $\cG$ is trimmable (though it is not minor-closed).
Let $\cA$ be the addable class of graphs with no minor a diamond (so $0< \rho_\cA < \infty$).   Then 
the conditions of Theorem~\ref{thm.smooth-likeGS} hold.
For, by an edge-version of the Erd\H{o}s-P\'osa theorem~\cite{rst2014}, 
there is an integer $f(k)$ such that by removing at most $f(k)$ edges we may obtain a graph in $\cA\,$: it follows that $|\cG_n| = O(n^{2f(k)}) \, |\cA_n|$,
and so $\rho_{\cG}=\rho_{\cA}$.

\smallskip

For our second new 
theorem here, Theorem~\ref{thm.smooth-Frag} concerning fragments, we slightly strengthen the conditions in Theorem~\ref{thm.smooth-likeGS}, but the above examples for Theorem~\ref{thm.smooth-likeGS} are also examples to which Theorem~\ref{thm.smooth-Frag} applies.
\begin{theorem} \label{thm.smooth-Frag}
Let the class $\cG$ of graphs be 
bridge-addable and closed under taking subgraphs,
with $0<\rho_\cG < \infty$.  Let $\cA$ be the class of graphs which are
free for $\cG$; and suppose that $\cA$ contains a graph which has a cycle, and that $\rho_{\cA}=\rho_{\cG} := \rho_0$. 
%
Then the conditions in Theorem~\ref{thm.smooth-likeGS}  hold, 
so we have the conclusions (a) - (d) of that theorem (including the specification of $\rho_2$).


Let $R_n \in_u \cG$. Then the fragment of $R_n$ converges in distribution to $BP(\cA, \rho_0)$;
%
%
and whp the fragment of the core of $R_n$ is the same as the core of the fragment of $R_n$, and they and the fragment of $R'_n \in_u \cG^{\delta \geq 2}$
 each converge in distribution to $BP(\cA^{\delta \geq 2}, \rho_2)$.
\end{theorem}
%
In this theorem it is straightforward to check
that the conditions in Theorem~\ref{thm.smooth-likeGS} hold, see the start of Section~\ref{sec.proofthm3}. 
Also, note that a graph $H$ is free for $\cG$ if and only if the disjoint union $G \cup H \in \cG$ whenever $G \in \cG$.

The first part of Theorem~\ref{thm.smooth-Frag} yields part (b) of Theorem~\ref{thm.smoothold} assuming that some graph with a cycle is free for $\cG$
 (and we have already seen that Theorem~\ref{thm.smooth-likeGS} extends the other parts of Theorem~\ref{thm.smoothold}, with this same assumption).
To see this, 
note that, letting  $\cA$ be the class of graphs which are
free for $\cG$, we have (i) if $\cG$ is 
addable minor-closed then $\cA$ is $\cG$, and (ii) if $\cG$ is $\cE^S$ then $\cA$ is $\cP$. 
%
%
The special case of the last part of Theorem~\ref{thm.smooth-Frag} when $\cG$ is an addable minor-closed class was first presented at the same time as the corresponding case of the last part of Theorem~\ref{thm.smooth-likeGS}, at RANDOM 2014.

 

We commented that for Theorem~\ref{thm.smooth-Frag} 
we strengthened the conditions in Theorem~\ref{thm.smooth-likeGS}.  The following example shows that we needed to do this.  Let $\cG$ be the trimmable class $\cP$ and let $\cA$ be the addable subclass of graphs in $\cP$ such that each component has at least 4 vertices in its core.  Since both $\cG$ and $\cA$ have growth constant $\gamma_\cP$, the conditions in Theorem~\ref{thm.smooth-likeGS} hold.  But for $R_n \inu \cG$, the probability that $\Frag(R_n)=K_1$ tends to $\rho_\cP$ (see~\cite{cmcdPAC2021}),
and $K_1$ is not in $\cA$, so the conclusions of Theorem~\ref{thm.smooth-Frag} do not all hold.


\medskip

\noindent
\emph{Plan of the rest of the paper}

The rest of the paper is essentially devoted to proving Theorems~\ref{thm.smooth-likeGS} and~\ref{thm.smooth-Frag}.
In the next three sections 
we present preliminary results, concerning `weak growth constants' in Section~\ref{sec.wgc}, the composition approach of Bender, Canfield and Richmond~\cite{bcr08} in Section~\ref{sec.bcr}, and an extended form of the Pendant Appearances Theorem from~\cite{msw05,msw06} in Section~\ref{sec.pendappthm}. 
The proof of Theorem~\ref{thm.smooth-likeGS} is completed in Section~\ref{sec.proofthm2}, and the proof of Theorem~\ref{thm.smooth-Frag} in Section~\ref{sec.proofthm3}. There are some brief concluding remarks in Section~\ref{sec.concl}.

\section{Weak growth constant} \label{sec.wgc}

In this section, we introduce the concept of a weak growth constant, which is implicit in~\cite{bcr08};
and we see in particular that if $\rho(\cG) < \rho(\cH)$ then the labelled product $\cG \otimes \cH$ has a weak growth constant if and only if $\cG$ does.

Given an infinite set $S \subseteq \N$, for each $n \in \N$ let $n^+ = \min\{m \in S : m \geq n\}$ and let $\gap_S(n)=n^+ -n$.  We say that a set $S \subseteq \N$ has \emph{sublinear gaps} if $S$ is infinite and $\gap_S(n)=o(n)$ as $n \to \infty$.
Thus $S$ has sublinear gaps if and only if for each $\eps>0$ there is an $n_0$ such that
for each $n \geq n_0$ there is an $n' \in S$ with $n \leq n' \leq n+ \eps n$. 
$|n-n'| < \eps n$.)

We may rephrase the above definition in terms of `slowly increasing' functions. Call a function or sequence $f: \N \to \N$ \emph{slowly increasing} if it is increasing and  
$(f(n+1)-f(n))/f(n) \to 0$ as $n \to \infty$.
For example $f(n)=n^k$ is slowly increasing for any fixed $k>0$, but $f(n)=2^n$ is not.
Given an infinite set $S$ of positive integers, let $f_S: \N \to S$ be the function or sequence enumerating $S$; that is
\[ f_S(n)=k \mbox{ if } |S \cap \{1,\ldots,k-1\}|=n-1 \mbox{ and } k \in S.\]
Then $S$ has sublinear gaps if and only if $f_S$ is slowly increasing. 

Let the graph class $\cG$ satisfy $0< \rho_{\cG}<\infty$. 
Given $0<\eta<1$, we say that 
 $n \in \N$ is $(1\!-\!\eta)$-\emph{rich} for $\cG$ if
\[ (|\cG_n|/n!)^{1/n} \geq (1 - \eta) \rho_{\cG}^{-1}.\]
Note that $\cG$ has a growth constant if and only if, for each $\eta>0$, all sufficiently large $n \in \N$ are $(1-\eta)$-rich for $\cG$.
We say that $\cG$ has a \emph{weak growth constant} if for each $\eta>0$ the set of $n \in \N$ which are $(1-\eta)$-rich for $\cG$
has sublinear gaps. 
%
If a class has a growth constant then clearly it has a weak growth constant. The converse does not hold in general: for example, cubic planar graphs do not have a growth constant (since they must have an even number of vertices) but they do have a weak growth constant
(see~\cite{blkm05}).
However, for a trimmable class $\cG$ the converse does hold, see Lemma~\ref{lem.trim-wgc-gc} below; and indeed when $\rho_{\cG}< e^{-1}$ the class $\cG$ must be smooth, see Lemmas~\ref{lem.Gtrim} and~\ref{lem.Ctrim}.

Next we discuss how having a weak growth constant, or growth constant, or being smooth, behaves under unions and products of graph classes.

\smallskip

\noindent
\emph{Growth constants and unions of graph classes}

First let us record some easily checked observations concerning a union of two graph classes $\cG$ and $\cH$.  It is easy to see that
$\rho(\cG \cup \cH) = \min \{ \rho_{\cG}, \rho_{\cH} \}$.
If $\rho_{\cG} \leq \rho_{\cH}$ and $\cG$ has a weak growth constant or growth constant then so does $\cG \cup \cH$. 
If $\rho_{\cG} < \rho_{\cH}$ then $\cG \cup \cH$ has a weak growth constant, growth constant, or is smooth if and only if $\cG$ has the same property.
\smallskip

\noindent
\emph{Growth constants and products of graph classes}

Now consider the (labelled) product $\cG \otimes \cH$ of two graph classes $\cG$ and $\cH$, see~\cite{fs09}.  This consists of all graphs with vertex set $V$ divided into disjoint sets $V=B \cup R$ where the blue set $B$ induces a graph in $\cG$, the red set $R$ induces a graph in $\cH$, and there are no edges between the blue and red vertices.  (We may for example have $B= \emptyset$ with the null graph in $\cG$ on $B$.) 
The exponential generating function of $\cG \otimes \cH$ is the product $G(x)H(x)$, where $G(x), H(x)$ are the exponential generating functions of $\cG$, $\cH$ respectively.  Thus 
\begin{equation} \label{eqn.rhotimes}
\rho(\cG \otimes \cH) = \min \{ \rho_{\cG}, \rho_{\cH} \}.
\end{equation}
The next lemma is straightforward to prove,
see for example P\'{o}lya and Szeg{o}~\cite{ps1998}. 


\begin{lemma} \label{lem.gcstimes}  Let $\cG$ and $\cH$ be graph classes, with $0< \rho_{\cG} < \infty$.
If $\rho_{\cG} \leq \rho_{\cH}$ and $\cG$ has a weak growth constant or a growth constant then so does $\cG \otimes \cH$.
If $\rho_{\cG}< \rho_{\cH}$ and $\cG$ is smooth then $\cG \otimes \cH$ is smooth.
\end{lemma}


%

When are there converse results?
Consider the example where $\cG$ is the set of paths with an even number of vertices, and $\cH$ consists of the null graph and the one vertex graph $K_1$.  Then clearly $\cG$ does not have a growth constant (and so is not smooth).  But $\cG \otimes \cH$ is smooth, indeed $|(\cG \otimes \cH)_n| = (1/2)n!$ for each $n \geq 2$.
Thus there is no converse result for the growth constant
or smoothness in Lemma~\ref{lem.gcstimes}.

Now consider the weak growth constant.  Again there is no converse result, as long as we allow $\rho_\cG=\rho_\cH$.
For example define $\cG$ as follows. For each $n \in \N$, let $\cG_n$ be the set $\cF_n$ of forests if $2^{2k-2} \leq n \leq 2^{2k-1}$
for some $k \in \N$, and let $\cG_n=\emptyset$ otherwise (and let $\cG$ contain the null graph). 
Then $\cG$ does not have a weak growth constant, but $\cG \otimes \cG$ does and indeed it 
has growth constant $\gamma_\cF=e$.
%
[We may see this as follows.
Since $\cG \subseteq \cF$ 
 and using~(\ref{eqn.rhotimes}), we have
$\rho(\cG \otimes \cG) \geq \rho(\cF \otimes \cF) = \rho(\cF) = 1/e$.
Now let $n \in \N$.  If $2^{2k-2} \leq n \leq 2^{2k-1}$ then $\cG_n =\cF_n$.
Now suppose that $2^{2k-1} < n < 2^{2k}$.  Let $n_1=\lfloor n/2 \rfloor$ and $n_2=\lceil n/2 \rceil$, so $2^{2k-2} \leq n_1 \leq n_2 \leq 2^{2k-1}$ and $n=n_1+n_2$.  Then
\[|(\cG \otimes \cG)_n| \geq \binom{n}{n_1} | \cF_{n_1}| |\cF_{n_2}| = n!\, \frac{|\cF_{n_1}|}{n_1!} \frac{\cF_{n_2}|}{n_2!} = (1+o(1))^n n! e^n\,.
\]
Thus $\cG \otimes \cG$ has growth constant $e$, as claimed.]

However, 
there \emph{is} a converse result for the weak growth constant when we insist that $\rho_\cG < \rho_\cH$.
This result will be important for us -- see the proofs of Lemmas~\ref{lem.Gtrim} and~\ref{lem.Ctrim}. 

\begin{lemma} \label{lem.wgctimesconv}
  Let $\cG$ and $\cH$ be graph classes with $0< \rho_{\cG}< \rho_{\cH}$, and suppose that $\cG \otimes \cH$ has a weak growth constant.  Then $\cG$ has a weak growth constant.
\end{lemma}

\begin{proof}
Given $0<\eps<1$, $0<\eta<1$ and $n$, let $P(\eps,\eta \,;n)$ be the proposition that there is no $k$ with $(1-\eps)n \leq k \leq n$ which is $(1-\eta)$-rich for $\cG$.  The main step in the proof of the lemma is to establish the following claim.

\noindent
{\bf Claim} \, Let $0<\eps<1$ and $0<\eta<1$.  Define $\delta$ 
\[ \delta = \tfrac12 \min \{ {1-( \rho_{\cG}/\rho_{\cH} )^{\eps}}, {1- (1-\eta)^{1-\eps}} \},  \]
and note that $0<\delta< \frac12$.
Then there exists $n_0=n_0(\eps,\eta)$ such that for each $n \geq n_0$, if $P(\eps,\eta\,;n)$ holds then $n$ is not $(1-\delta)$-rich for $\cG \otimes \cH$.

\begin{proof}[Proof of Claim]
Suppose that $P(\eps,\eta\,;n)$ holds.  Then $|(\cG \otimes \cH)_n|$ is at most
\[ \sum_{k < (1\!-\!\eps)n}\!\binom{n}{k} |\cG_k| \, |\cH_{n-k}| \; +
\sum_{(1\!-\!\eps)n \leq k \leq n}\!\binom{n}{k} (1\!-\!\eta)^k \rho_{\cG}^{-k} k! \, |\cH_{n-k}|. \]
The first sum equals
\begin{eqnarray*}
  &&
     n! \sum_{k < (1\!-\!\eps)n} \frac{|\cG_k|}{k!} \frac{|\cH_{n-k}|}{(n-k)!}\\
  & \leq &
     n! \: n \, (\rho_{\cG}+o(1))^{-(1-\eps)n}  (\rho_{\cH}+o(1))^{-\eps n}\\
  & = &
    (1+o(1))^n n! \, \rho_{\cG}^{-n} (\rho_{\cG}/\rho_{\cH})^{\eps n}\\
  &=&
     o((1-\delta)^n) \cdot n! \, \rho_{\cG}^{-n}.
\end{eqnarray*}
The second sum is at most
\[   (1+o(1))^n (1-\eta)^{(1-\eps)n} \cdot n! \, \rho_{\cG}^{-n} = o((1-\delta)^n) \cdot n! \, \rho_{\cG}^{-n}. \]
Hence (recalling that $\rho(\cG \otimes \cH)=\rho_{\cG}$), if $n$ is sufficiently large and $P(\eps,\eta\,;n)$ holds then
\[ |(\cG \otimes \cH)_n| < (1-\delta)^n \cdot n! \, \rho(\cG \otimes \cH)^{-n}, \]
and so $n$ is not $(1\!-\!\delta)$-rich for $\cG \otimes \cH$, which completes the proof of the claim.
\end{proof}
Now we use the claim to complete the proof of the lemma.
Suppose that $\cG$ does not have
a weak growth constant.
Then there exist  $0<\eps< \frac12$ and $0<\eta<1$ such that for every $n_1$ there is an $n \geq n_1$ for which $P(2\eps,\eta\,;n)$ holds.  Observe that if $P(2\eps,\eta\,;n)$ holds and $(1-\eps)n \leq m \leq n$ then $P(\eps,\eta\,;m)$ holds.
Fix such $\eps$ and $\eta$\,;
and let $\delta$ and $n_0$ 
be as in the Claim.  Let $n_1 \geq (1-\eps)^{-1} n_0$.



Now consider $n \geq n_1$ such that $P(2\eps,\eta\,;n)$ holds.  Let $(1-\eps)n \leq m \leq n$.  Then $m \geq n_0$ and $P(\eps,\eta\,;m)$ holds; and so, by the Claim, $m$ is not $(1-\delta)$-rich for $\cG \otimes \cH$. 
It follows that $\cG \otimes \cH$ does not have a weak growth constant, which completes the proof of Lemma~\ref{lem.wgctimesconv}.
\end{proof}
\smallskip

\noindent
\emph{Growth constants and trimmable graph classes}

The final lemma in this subsection shows that, 
if a trimmable graph class $\cG$ has a weak growth constant, then it has a growth constant.  Note that if a class is trimmable then it is closed under adding a leaf. We shall not need this lemma here: it is included for completeness.  

\begin{lemma} \label{lem.trim-wgc-gc}
Let the class $\cG$ of graphs be closed under adding a leaf.
If $\cG$ has a weak growth constant, then it has a growth constant.  Similarly, if the class $\cC$ of connected graphs in $\cG$ has a weak growth constant, then it has a growth constant, and so also $\cG$ has a growth constant.
\end{lemma}

\begin{proof}
Let  
$\cA$ be either $\cG$ or $\cC$.
Write $\gamma$ for $\rho_{\cA}^{-1}$, and suppose that $\cA$ has weak growth constant $\gamma$.
Let $\hat{\cA}$ be the class of graphs in $\cA$ which have a leaf.
It is easy to see that
$|\hat{\cA}_{n+1}| \geq n |\cA_n|$, and it follows that $\hat{\cA}$ has weak growth constant $\gamma$.

Let $0<\eps< \gamma$.  Let $n_1, n_2, \ldots$ be a slowly increasing sequence such that
$(|\hat{\cA}_{n_k}|/n_k !)^{1/n_k} \geq \gamma - \eps$ for each $k$.
Given $n \geq n_1$ let $k=k(n)$ be such that $n_k \leq n < n_{k+1}$. 
Let $n \geq n_1$. 
Pick an $n_k$-subset $W$ of $[n]$, and a
graph $G \in \hat{\cA}$ on $W$. Pick a leaf $v$ of $G$, with adjacent vertex $u$.  Pick an ordering $v_1,\ldots, v_{n-n_k}$ of $[n] \setminus W$, and subdivide the edge $vu$ so that it becomes the path $v,v_1,\ldots, v_{n-n_k},u$.
This gives at least
\[ \binom{n}{n_k} \,|\hat{\cA}_{n_k}| \,(n-n_k)! = \frac{n!}{n_k!}\, |\hat{\cA}_{n_k}|
\]
constructions of graphs in $\hat{\cA}_n$. Each graph can be constructed at most $n$ times (since we know $k$ and so once we guess the leaf $v$ we know the whole construction). 
Thus \[  |\hat{\cA}_n| \geq \frac{n!}{n_k!}\, |\hat{\cA}_{n_k}| \cdot \frac1{n}  \]
and so
\begin{eqnarray*}
 \left(\frac{|\hat{\cA}_n|}{n!}\right)^{\frac{1}{n}}
& \geq &
\left(\frac{|\hat{\cA}_{n_k}|}{n_k!} \right)^{\frac{1}{n}} \cdot n^{-\frac1{n}}
\; \geq \;
  (\gamma - \eps)^{\frac{n_k}{n}} \, n^{-\frac1{n}}\\
& = &
 (1+o(1))\, (\gamma - \eps)^{\frac{n_k}{n}} \; = \; \gamma-\eps +o(1)
\end{eqnarray*}
  as $n \to \infty$.  It follows that $\hat{\cA}$ has growth constant $\gamma$, and thus so does~$\cA$.
\end{proof}



\section{Composition of graph classes and smoothness}
\label{sec.bcr}

A key idea in our proofs here is that the graphs in a trimmable class $\cG$ may be constructed by first choosing the core and then rooting trees at the vertices of the core, so if the numbers of graphs in $\cGd_n$ behave reasonably well as $n$ varies (for example if $\cGd$ has a weak growth constant), then we may deduce smoothness for $\cG$. (We may handle tree components separately.)
We now describe the composition approach of Bender, Canfield and Richmond~\cite{bcr08}, mentioned above. In particular, we show that we can change the assumption about which graph class has a weak growth constant: this will be needed later, in the proofs of Lemmas~\ref{lem.Gtrim} and~\ref{lem.Ctrim}.

Let $\cEr$ be a non-empty class of (vertex-) rooted connected graphs 
(which each must have at least one vertex),
and let $\cD$ be a class of graphs (containing the null graph).
The \emph{rooted composition} $\cD \circ \cEr$ of $\cD$ with $\cEr$ is the set of multiply rooted graphs $G$ which may be obtained as follows.
We start with a graph $D \in \cD$, let $V(D)$ be the root set (coloured red say), and root disjoint graphs in $\cEr$ at the vertices of $D$, one at each vertex.
We call the graph $D$
the $\cD$-\emph{core} of $G$,
and denote the order of the $\cD$-core of $G$ by $v_{\cD}(G)$.  The null graph in $\cD$ gives rise to the null graph in $\cD \circ \cEr$.

We shall be interested mostly in the case when the root set tell us nothing that we cannot see for ourselves, and so we can ignore which vertices are coloured red; and in particular we are interested in the case when $\cEr$ is the class of rooted trees and each graph in $\cD$ has minimum degree at least~2.
Suppose that each graph $G$ in the class $\cG$ has a unique induced subgraph $D$ in $\cD$ (which we shall call its $\cD$-\emph{core}) such that $G$ may be obtained by starting with $D$ and rooting disjoint graphs in $\cEr$ at the vertices of $D$, and then forgetting the root set; and each graph obtained in this way is in $\cG$.
Then we say that $\cG$ is the \emph{(unrooted) composition} of $\cD$ with $\cEr$.  We still write $\cG= \cD \circ \cEr$, though now the graphs in $\cG$ are uncoloured, and call $D$ the $\cD$-core of $G$.

Suppose that $\cG= \cD \circ \cEr$, and let $n, k \in \N$.  Then $n! \, [x^n] \cEr(x)^k /k!$ is the number of graphs on $[n]$ consisting of $k$ rooted components in $\cEr$.  Thus
\[ | \{ G \in \cG_n: v_{\cD}(G)=k\}| =  n! \, [x^n] \, |\cD_k| \, \cEr(x)^k /k!\]
and so $G(x)=D(\Er(x))$.
Now suppose that  $0<\rho_{\cG}< \rho_{\cEr}$.
Then $\rho_{\cD}>0$ since  $\rho_{\cG}>0$ and $\cEr$ is non-empty.
If we had $\Er(s)< \rho_{\cD}$ for each $s$ with $0<s< \rho_{\cEr}$, then we would have $\rho_{\cG} \geq \rho_{\cEr}$: hence $\rho_{\cD}< \infty$, and $\Er(s) \geq \rho_{\cD}$ for some $s$ with $0<s<\rho_{\cEr}$.  But  $\cEr(x)$ is continuous and strictly increasing for $0<x< \rho_{\cEr}$, so there is a unique $s>0$ such that $\Er(s)=\rho_{\cD}$; and $s$ is clearly $\rho_{\cG}$.
Observe that all we are using at this point about $\cD$ is the value of $\rho_\cD$.
%

Suppose that the graph class $\cG$ is trimmable (and recall that $\cG$ contains the null graph).  Then $\cG$ may be presented as the unrooted composition
\begin{equation} \label{eqn.Gcomp}
   \cG = (\cGd o \, \cTr) \otimes \cFor
\end{equation}
where $\cTr$ is the class of rooted trees (recall that $\cFor$ is the class of forests);
and similarly the class $\cC$ of connected graphs in $\cG$ may be written as 
\begin{equation} \label{eqn.Ccomp}
   \cC = (\cCd o \, \cTr) \cup \cT .
\end{equation}

We now present a version of Theorem 1 of Bender, Canfield and Richmond~\cite{bcr08}, stated in terms of graph classes.
%
We call a class $\cE$ of graphs \emph{aperiodic} if the greatest common divisor of the terms $i-j$ where $\cE_i$ and $\cE_j$ are non-empty is 1.

\begin{lemma}  (Bender, Canfield and Richmond~\cite{bcr08}) \label{lem.bcr}
  Let $\cEr$ be an aperiodic class of rooted connected graphs; let $\cD$ be a class of graphs; 
and let $\cG$ be the rooted composition $\cD o \, \cEr$.  Suppose that $0<\rho_{\cG}< \rho_{\cEr} < \infty$, and $\cD$ has a weak growth constant.  Then $\cG$ is smooth.
\end{lemma}

  For the special case of Lemma~\ref{lem.bcr} when $\cEr$ is the class $\cTr$ of rooted trees, and each graph in the class $\cD$ has minimum degree at least 2 (as in~(\ref{eqn.Gcomp}) and~(\ref{eqn.Ccomp}), with $\cD$ as $\cGd$ or $\cCd$), 
Lemma~\ref{lem.bcr} tells us that $\cG$ is smooth, but
we can say more and in particular we have equation~(\ref{eqn.far1}). Further, we can use~(\ref{eqn.far1}) to give a straightforward combinatorial proof that $\cG$ is smooth (without using Lemma~\ref{lem.bcr}), and for completeness we give such a proof.

\begin{lemma} \label{lem.afterbcr}
Let each graph in the class $\cD$ have minimum degree at least 2, and let $\cG$ be the unrooted composition $\cD \,o\, \cTr$.  Suppose that $0<\rho_\cG < e^{-1}$ and $\cD$ has a weak growth constant.  Then $\cG$ is smooth, 
$\rho_\cD$ is the unique root $x \in (0,1)$ to $x e^{-x}=\rho_\cG$;
and for any $\eps>0$
\begin{equation} \label{eqn.far1}
|\{G \in \cG_n:  | v_{\cD}(G) - (1-\rho_\cD) n| > \eps n \}| = e^{-\Omega(n)}\, |\cG_n|  \;\; \mbox{ as } n \to \infty. 
\end{equation}
\end{lemma}
\begin{proof}
%
Recall from earlier in this section that $0< \rho_{\cD}< \infty$, and $\rho_\cG$ is the unique value $s \in (0,1/e)$ such that  $\cTr(s) = \rho_\cD$.  But $s=T^{\bullet}(s)\, e^{-T^{\bullet}(s)}$, and so
\[ \rho_{\cG}= s=T^{\bullet}(s)\, e^{-T^{\bullet}(s)}=\rho_{\cD} \, e^{-\rho_{\cD}}. \]
We may for example read the rest of the proof of~(\ref{eqn.far1}) 
from the proof of Theorem 2.4 in~\cite{cmcd2013}, as follows.
%
  Let $r(n) = |\cD_n|\, \rho_{\cD}^n/n! \leq (1+o(1))^n$, 
%
  and let $s(n) = (n/e)^n / n! \leq 1$. 
Recall that the number of forests of rooted trees on $[n]$ with a given set of $k$ root vertices is $kn^{n-1-k}$, see for example~\cite{Moon1970}. 
For $3 \leq k \leq n-1$, writing $\kappa = k/n$ we have
 \begin{eqnarray*}
   | \{ G \in \cG_n: v_{\cD}(G)=k\}|
  & = &
   \binom{n}{k} \ |\cD_k| \  k n^{n-1-k} \\
  & = &
   n! \ \frac{|\cD_k| }{k!} \ \frac{k}{n} \ \left(\frac{n}{n-k}\right)^{n-k}
   \ \frac{(n-k)^{n-k}}{(n-k)!}\\
  & = &
  n! \ r(k) \ \rho_{\cD}^{-k} \ \frac{k}{n} \left(\frac{e}{1-k/n}\right)^{n-k} s(n-k)\\
  & \leq &
  (1+o(1))^n \ n! \ \rho_{\cD}^{-n} \
  \left(\frac{e \rho_{\cD}}{1-\kappa}\right)^{(1-\kappa)n} . 
\end{eqnarray*}
  %
Hence
\begin{equation} \label{eqn.eta}
   | \{ G \in \cG_n: v_{\cD}(G)=k\}| \leq (1+o(1))^n \
  n! \ \rho_{\cD}^{-n} \  \left(h(1-\kappa)\right)^{n}
\end{equation}
where $h(x) = (e \rho_{\cD} / x)^x$ for $x>0$.
Now the function $h(x)$ strictly increases up to $x=\rho_{\cD}$, where it has value $e^{\rho_{\cD}}$, and strictly decreases above $\rho_{\cD}$.
  Thus, given $\eps>0$, there exists $\eta>0$ such that $h(x) \leq (1-\eta) e^{\rho_{\cD}}$ for $|x-\rho_{\cD}| \geq \eps$.
  
We may now see that $\rho_\cD<1$.  First, if $\rho_\cD>1$ then by~(\ref{eqn.eta}) and the bound on $h(x)$ (summing over the possible values of $k$), there exists $\eta>0$ such that
\[ |\cG_n| \leq (1+o(1))^n \, n!\, \rho_\cD^{-n}  \, (1-\eta)^n e^{\rho_\cD\,n} = (1-\eta+o(1))^n \, n!\, \rho_\cG^{-n} : \]
but this is false, so we must have $\rho_\cD \leq 1$.  Also, if $\rho_\cD =1$ then by~(\ref{eqn.eta}) we have
$|\cG_n| \leq (1+o(1))^n \, n!\,e^n$, so $\rho_\cG \geq e^{-1}$: but this contradicts the assumption that $\rho_\cG < e^{-1}$, so $\rho_\cD \neq 1$. We have now shown that $0<\rho_\cD<1$, and we saw earlier that $\rho_\cD$ satisfies $xe^{-x}=\rho_\cG$.  But the function $f(x)=xe^{-x}$ is strictly increasing for $x \in (0,1)$,  so $\rho_\cD$ is the unique root in $(0,1)$ to $xe^{-x}=\rho_\cG$.

Again using~(\ref{eqn.eta}) and the bound on $h(x)$, when $|k-(1-\rho_{\cD}) n| \geq \eps n$
\begin{eqnarray*}
| \{ G \in \cG_n: v_{\cD}(G)=k\}| 
&\leq &
(1-\eta+o(1))^n \ n! \ (\rho_{\cD}^{-1} e^{\rho_{\cD}})^{n}\\
& = &
(1-\eta+o(1))^n  \ n! \ \rho_{\cG}^{-n}.
\end{eqnarray*}
  Since $\cG$ has growth constant $\rho_{\cG}^{-1}$, 
equation~(\ref{eqn.far1}) follows.

Now we use~(\ref{eqn.far1}) to give a direct proof of smoothness (without using Lemma~\ref{lem.bcr}).  For $n, k \in \N$ with $n > k$ and $\cD_k \neq \emptyset$, let
\[ r_{n,k}= \frac{ n\, |\{G \in \cG_{n-1}: v_\cD(G)=k \}|}{|\{G \in \cG_{n}: v_\cD(G)=k \}|}\,; \]
and note that
\[ r_{n,k} = n \frac{\binom{n-1}{k}}{\binom{n}{k}} \frac{(n-1)^{n-k-2}}{n^{n-k-1}} = \tfrac{n-k}{n} \,(1-\tfrac1{n})^{n-k-2}\,.\]
But for $k \sim (1-\rho_\cD)n$, we have $\frac{n-k}{n} \sim \rho_\cD$ and 
\[ (1-\tfrac1{n})^{n-k-2} = \big((1-\tfrac1{n})^{n}\big)^{(1-\frac{k+2}{n})} 
\sim e^{-\rho_\cD},\]
and so when $\cD_k \neq \emptyset$ we see that
$r_{n,k} \sim \rho_\cD\, e^{-\rho_\cD} = \rho_\cG$.
Hence, for any $0<\eps<1$, there exists $\delta>0$ and $n_0$ such that, letting
\[ I_n = \{k \in \N : (|k-(1-\rho_\cD) n| \leq \delta n) \land (\cD_k \neq \emptyset) \}\;\; \mbox{ for } n \in \N\,,\]
if $n \geq n_0$ and $k \in I_n$ then $r_{n,k}=(1 \pm \eps) \rho_\cG$ (where this means that $(1-\eps)\rho_\cG \leq r_{n,k} \leq (1+ \eps)\rho_\cG$). 
Thus by~(\ref{eqn.far1}) (for $n-1$ and $n$)
\begin{eqnarray*}
\frac{n\,|\cG_{n-1}|}{|\cG_n|} &=&
(1 \pm e^{-\Omega(n)})\,
\frac{ n\, \sum_{k \in I_n} |\{G \in \cG_{n-1}: v_\cD(G) =k\}|}
{\sum_{k \in I_n} |\{G \in \cG_{n}: v_\cD(G) =k\}|}\\
&=&
(1 \pm e^{-\Omega(n)})\,
\frac{ \sum_{k \in I_n} r_{n,k}\,
|\{G \in \cG_{n}: v_\cD(G) =k\}|}
{\sum_{k \in I_n} |\{G \in \cG_{n}: v_\cD(G) =k\}|}\\
&=&
(1 \pm e^{-\Omega(n)})\, (1 \pm \eps)\, \rho_\cG\,.
\end{eqnarray*}
Thus
$r_n = \frac{n\,|\cG_{n-1}|}{|\cG_n|} \to \rho_\cG$ as $n \to \infty$; that is, $\cG$ is smooth.  This completes the proof of Lemma~\ref{lem.afterbcr}.
%
\end{proof}

We shall want to apply a result like Lemma~\ref{lem.bcr} or Lemma~\ref{lem.afterbcr} also in the case when we know that $\cG$ rather than $\cD$ has a weak growth constant.  The following lemma will allow us to do that.

\begin{lemma} \label{lem.Gwgc}
  In Lemma~\ref{lem.bcr} or Lemma~\ref{lem.afterbcr}, suppose that we assume that $\cG$ rather than $\cD$ has a weak growth constant, and that $\rho_\cD <1$.  Then in fact $\cD$ has a weak growth constant.
\end{lemma}

\begin{proof}
Recall that $\rho_\cD>0$ (since $\rho(\cG)>0$ and $\cEr$ is not empty).
%
Let $\hat{\cD}$ be a class of graphs such that  $\hat{\cD}_n \supseteq \cD_n$, 
and
\[ \rho_{\cD}^{-n} n! \leq |\hat{\cD}_n| \leq \max \{ |\cD_n|, (\rho_{\cD}^{-n} +1) n! \}\]
for all sufficiently large $n$.
Then $\hat{\cD}$ has growth constant $\rho_{\cD}^{-1}$.
We may apply 
Lemma~\ref{lem.bcr} with $\hat{\cD}$ in place of $\cD$, to $\hat{\cG} := \hat{\cD} \circ \cEr$.  Write $\hat{\alpha}$ 
for the corresponding positive constant in~(\ref{eqn.far1}). 
Observe that $\rho_{\hat{\cG}}= \rho_{\cG}$ since $\rho_{\hat{\cD}}= \rho_{\cD}$.

Fix $0<\eta<1$. 
Given a sufficiently large positive integer $t_0$, we want to show that there is a $t_1$ 
with $|t_1-t_0| < \eta t_0$ which is  $(1-\eta)$-rich for $\cD$
(since that will imply that $\cD$ has a weak growth constant, as required). 
  Set $\delta_1= \frac12 \hat{\alpha} \eta$.
By~(\ref{eqn.far1}) applied to~$\hat{\cG}$, there exists $\delta_2>0$ such that
\begin{equation} \label{eqn.far2}
  |\{G \in \hat{\cG}_n:  | v_{\hat{\cD}}(G) - \hat{\alpha} n| > \delta_1 n \}|
  = O((1-\delta_2)^{n}) \; | \hat{\cG}_n|.
\end{equation}
Finally, let $\delta_3 >0$ be sufficiently small that $\delta_3 < \delta_2$ and
\begin{equation} \label{eqn.delta3}
 (1-\eta)^{\hat{\alpha}-\delta_1} (1-\delta_3)^{-1} < 1.
 \end{equation}
Let $S$ be the set of $n \in \N$ which are $(1-\delta_3)$-rich for $\cG$, and set $d(n)= \gap_S(n)$.  Then $d(n) = o(n)$ since $\cG$ has a weak growth constant.

Let $t_0$ be a (large) positive integer.  Let $n_0= \lfloor t_0/\hat{\alpha}\rfloor$.
By our choice of $d(n)$, there is an $n_1$ with $n_0 \leq n_1 \leq n_0+d(n_0)$ such that $n_1$ is $(1-\delta_3)$-rich for $\cG$.
Let $t_1$ maximise (over $t$) the count $|\{G \in \cG_{n_1}: v_{\cD}(G)=t\}|$.
We shall see that, if $t_0$ is sufficiently large, then $t_1$ is as required.
We do this in three steps, showing (a) that  $|t_1 - \hat{\alpha} n_1| \leq \delta_1 n_1$, then (b) that $t_1$ is $(1-\eta)$-rich for $\cD$, and finally (c)  that $|t_1 -t_0| < \eta t_0$.
Observe first that 
\begin{equation} \label{eqn.av}
 |\{G \in \cG_{n_1}: v_{\cD}(G)=t_1\}| \geq |\cG_{n_1}|/n_1
\end{equation}
since $v_\cD(G)$ here can take at most $n_1$ values.

(a) Suppose for a contradiction that $|t_1 - \hat{\alpha} n_1| > \delta_1 n_1$.  Then using~(\ref{eqn.far2}) and 
the facts that $\rho_{\hat{\cG}}=\rho_{\cG}$, $\, \delta_3<\delta_2$ and $n_1$ is $(1-\delta_3)$-rich for $\cG$, 
we have
\begin{eqnarray*}
|\{G \in \cG_{n_1}: v_{\cD}(G)=t_1\}|
  & \leq &
    |\{G \in \hat{\cG}_{n_1}: v_{\hat{D}}(G)=t_1\}|\\
  & = &
    O((1-\delta_2)^{n_1}) \, |\hat{\cG}_{n_1}| \\
  & \leq &
    (1- \delta_2 +o(1))^{n_1} \rho_{\cG}^{-n_1} n_1!\\
& = & e^{-\Omega(n_1)} (1-\delta_3)^{n_1} \rho_{\cG}^{-n_1} n_1!\\
  & \leq &
    e^{-\Omega(n_1)} |{\cG}_{n_1}|\,.  
\end{eqnarray*}
Here and later the $o(1)$ term is as $t_0$ (and so $n_0$ and $n_1$) tend to $\infty$.
We have just seen that 
\begin{equation} \label{eqn.expsmall}
|\{G \in \cG_{n_1}: v_{\cD}(G)=t_1\}| =e^{-\Omega(n_1)} |\cG_{n_1}|\,.
\end{equation}
Hence
\[ |\{G \in \cG_{n_1}: v_{\cD}(G)=t_1\}| 
 < |\cG_{n_1}|/n_1\]
if $t_0$ and thus $n_1$ is sufficiently large.  But this contradicts~(\ref{eqn.av}),
which completes the proof of (a).

(b) 
Suppose for a contradiction that $t_1$ is not $(1-\eta)$-rich for $\cD$. 
Let $N$ be the number of graphs on $[n_1]$ consisting of $t_1$ rooted components in $\cEr$ (so $N / n_1!\,$ is the coefficient of $x^{n_1}$ in $\cEr(x)^{t_1}/t_1!$). Then, since $n_1$ is $(1-\delta_3)$-rich for $\cG$,  and $t_1 \geq (\hat{\alpha} - \delta_1) n_1$, and finally using~(\ref{eqn.delta3})
\begin{eqnarray*}
 && |\{G \in \cG_{n_1}: v_{\cD}(G)=t_1\}|\\
 & = &
  |\cD_{t_1}| \cdot N \\
 & \leq &
  (1-\eta +o(1))^{t_1} \rho_\cD^{-t_1} t_1! \cdot N \\
 & \leq &
  (1- \eta+o(1))^{t_1} \; |\{G \in \hat{\cG}_{n_1}: v_{\hat{\cD}}(G)=t_1\}|\\
 & \leq &
  (1- \eta+o(1))^{t_1} \; |\hat{\cG}_{n_1}|\\
 & \leq &
  (1- \eta+o(1))^{t_1} \; (1-\delta_3)^{-n_1} |{\cG}_{n_1}|\\
 & = &
(1+o(1))^{n_1} \; \left( (1-\eta)^{\hat{\alpha} - \delta_1} (1-\delta_3)^{-1} \right)^{n_1}  |\cG_{n_1}|\\
  & = &
  e^{-\Omega(n_1)} \; |\cG_{n_1}|\,.
\end{eqnarray*}
%
%
Thus we have again deduced~(\ref{eqn.expsmall}),
which contradicts equation~(\ref{eqn.av}) when $t_0$ is sufficiently large; and
this completes the proof of (b).

(c) By~(a) we have
\[ |t_1- \hat{\alpha} n_1| \leq \delta_1 n_1 \leq \delta_1 (n_0 + d(n_0))\,;\]
and so, since $\delta_1 / \hat{\alpha} < \eta$,
\begin{eqnarray*}
 |t_1-t_0| & \leq &
  |t_1- \hat{\alpha} n_1| + | \hat{\alpha} n_1- \hat{\alpha} n_0| + |\hat{\alpha} n_0 - t_0| \\
  & \leq &
  \delta_1(n_0 + d(n_0)) + \hat{\alpha}d(n_0)+\hat{\alpha}\\
  & \leq &
  ((\delta_1/\hat{\alpha}) +o(1)) t_0  \;\;   < \; \eta t_0
\end{eqnarray*}
when $t_0$ is sufficiently large; and we have proved (c).

We have now seen that, when $t_0$ is sufficiently large, there is a $(1- \eta)$-rich integer $t_1$ with $|t_1 -t_0|  < \eta t_0$.
This completes the proof that $\cD$ has a weak growth constant, and thus completes the proof of Lemma~\ref{lem.Gwgc}.
\end{proof}


\section{Pendant Appearances}
\label{sec.pendappthm}

We shall need a slight extension of the `Pendant Appearances Theorem' (from~\cite{msw05,msw06}) in the proofs of Lemmas~\ref{lem.Gsmooth} and~\ref{lem.Csmooth} 
 below. Given a vertex-rooted connected graph $\Hb$, we say that a graph $G$ has a \emph{pendant appearance} of $\Hb$ if 
$G$ has a bridge $e$ such that one component of $G-e$ is a copy of $H$ and $e$ is incident to the root vertex of $\Hb$.  (In~\cite{msw05} we considered graphs on subsets of $\N$, and insisted that the increasing bijection was an isomorphism between the component of $G-e$ and $H$: we no longer do this.)


  Call a connected rooted graph $\Hb$ \emph{attachable} to $\cG$ if whenever we have a graph $G$ in $\cG$ and a disjoint copy of $H$, and we add an edge between a vertex in $G$ and the root of $\Hb$, then the resulting graph (which has a pendant appearance of $\Hb$) must be in $\cG$.
%
For an addable minor-closed class $\cG$, every connected rooted graph in $\cG$ is attachable in $\cG$; and
for a surface class $\cE^S$, the attachable graphs are the connected rooted planar graphs.

In the original version of the Pendant Appearances Theorem~\cite{msw05,msw06}, we assume that $\cG$ has a growth constant.
We need a slight generalisation of this result, Lemma~\ref{thm.pendapp-nogc}, in which we assume only that $0<\rho_\cG < \infty$. For recent stronger results see~\cite{cmcdPAC2021}.

\begin{lemma} \label{thm.pendapp-nogc}
  Let $\cG$ be a class of graphs with $0< \rho_{\cG} < \infty$, and let the connected rooted graph $H$ be attachable to $\cG$. Then there exists $\alpha>0$ such that the following holds. Let $\cH$ denote the class of graphs in $\cG$ with at most $\alpha n$ pendant appearances of $H$. Then $\rho_{\cH}> \rho_{\cG}$.
\end{lemma}

Lemma~\ref{thm.pendapp-nogc} can be proved in a very similar way to the original version of the Pendant Appearances Theorem~\cite{msw05,msw06}.
%
Of course $\rho_{\cH} \geq \rho_{\cG}$, so we suppose for a contradiction that $\rho_{\cH} = \rho_{\cG}$.
The idea of the proof  is to show that, 
when $n$ is large and $|\cH_n|$ is close to $\rho_{\cH}^{-n} n!$, we can construct ``too many'' graphs in $\cG$ on $(1+\delta)n$ vertices.  Essentially this is done by putting an $n$-vertex graph $G \in \cH$ on a subset of the vertices, and  attaching linearly many copies of $H$ using the other vertices:
the fact that $G$ has `few' pendant appearances of $H$ limits the amount of double-counting involved.


\section{Proof of Theorem~\ref{thm.smooth-likeGS}}
 \label{sec.proofthm2}
 
This section has two subsections.  In the first, we prove the first three parts of Theorem~\ref{thm.smooth-likeGS}; concerning smoothness of $\cG$ and $\cC$, growth constants for $\cGd$ and $\cCd$, and the order of the core.  In the second (longer) subsection, we prove the final part of the theorem.
 
\subsection{Proof of Theorem~\ref{thm.smooth-likeGS} parts (a), (b), (c)}

The two lemmas in this subsection quickly yield  parts (a), (b) and (c) of Theorem~\ref{thm.smooth-likeGS}.
The first lemma concerns a trimmable graph class, and the similar second lemma concerns the connected graphs in such a class.
Observe that when $\cG$ is a trimmable class of graphs and $\cC$ is the class of connected graphs in $\cG$, then $\cC$ contains all trees, and so $\rho_\cG \leq \rho_{\cC} \leq e^{-1}$.

%
\begin{lemma} \label{lem.Gtrim}
  Let $\cG$ be a trimmable class of graphs with $\rho_0:= \rho_{\cG} < e^{-1}$, and suppose that $\cG$ has a weak growth constant.  Then $\cG$ is smooth.  Further, 
$\rho_2:= \rho(\cGd)$ is the unique root $x \in (0,1)$  to $x e^{-x} = \rho_0\, $; $\cGd$ has weak growth constant $\rho_2^{-1}$; and for each $\eps>0$
\[  |\{ G \in \cG_n : |\,\core(G)- (1-\rho_2) n\,| > \eps n \}| = e^{-\Omega(n)}\, | \cG_n|. \]
\end{lemma}
\begin{proof}
We saw in~(\ref{eqn.Gcomp}) that
\[   \cG = (\cGd o \, \cTr) \otimes \cFor \]
(with no rooting).
Since $\rho_{\cG}< \rho_{\cF}$, by 
(\ref{eqn.rhotimes})  
\begin{equation} \label{eqn.Gtrim1}
   \rho(\cGd o \, \cTr) = \rho_{\cG} < \rho_{\cF}.
\end{equation}
Hence, by Lemma~\ref{lem.wgctimesconv}, since $\cG$ has a weak growth constant it follows that $\cGd o \, \cTr$ has a weak growth constant.
Thus by~(\ref{eqn.Gtrim1}) and Lemma~\ref{lem.Gwgc}, $\cGd$ has a weak growth constant. Hence by Lemma~\ref{lem.afterbcr},  $\cGd o \, \cTr$  is smooth; and so $\cG$ is smooth, by~(\ref{eqn.Gtrim1}) and Lemma~\ref{lem.gcstimes}.  Further, since $\core(G)=v_\cD(G)$ when $\cD=\cG^{\delta \geq 2}$, we may obtain the last sentence of the lemma directly from Lemma~\ref{lem.afterbcr}.
\end{proof}

\begin{lemma} \label{lem.Ctrim}
  Let $\cG$ be a trimmable class of graphs, let $\cC$ be the class of connected graphs in $\cG$, 
and assume that $\rho_\cC< e^{-1}$ and $\cC$ has a weak growth constant.
Then $\cC$ is smooth.
Further, $\cCd$ has a weak growth constant; $\rho_2:= \rho(\cCd)$ is the unique root $x \in (0,1)$ to $x e^{-x} = \rho_\cC$; and for each $\eps>0$
\[  |\{ G \in \cC_n : | \core(G)- (1-\rho_2) n| > \eps n \}| = e^{-\Omega(n)} | \cC_n|. \]
\end{lemma}

\begin{proof}
We argue as in the proof of Lemma~\ref{lem.Gtrim}.
 By equation~(\ref{eqn.Ccomp})
\[ \cC = (\cCd \circ \cTr) \cup \cT\]
(with no rooting).  Since $\rho_{\cT} > \rho_{\cC}$, from the discussion on 
unions of graph classes in Section~\ref{sec.wgc},
we see that $\cCd \circ \cTr$ has a weak growth constant, and 
to show that $\cC$ is smooth it suffices to show that $\cCd \circ \cTr$ is smooth.  But
\[  \rho(\cCd \circ \cTr) = \rho_{\cC} < e^{-1}= \rho_{\cTr},  \]
 and so by Lemma~\ref{lem.Gwgc}, $\cC^{\delta \geq 2}$ has a weak growth constant.
  Hence the required smoothness for $\cC$, and the rest of the last sentence of the
  lemma, follows from Lemma~\ref{lem.afterbcr}.
%
\end{proof}


  We may now quickly prove the first three parts of Theorem~\ref{thm.smooth-likeGS}.
\begin{proof}[Completing the proof of Theorem~\ref{thm.smooth-likeGS}  parts (a), (b), (c)]
Since $\cA$ is addable and $0<\rho_\cA < \infty$, it now follows that $\cA$ has a growth constant.  Since $\cA \subseteq \cG$ and $\rho_\cA=\rho_\cG$, it follows that $\cG$ has a growth constant.  In fact $\rho_\cA < e^{-1}$ since $\cA$ is addable and strictly contains $\cF$.
Now parts (a), (b), (c) for $\cG$ follow by Lemma~\ref{lem.Gtrim}.
Since $\cA$ is bridge-addable, by~(\ref{eqn.ba}) we have $|\cC_n| \geq |\cA_n|/e$; and so $\cC$ also has growth constant $\rho_\cA^{-1}$.  Thus parts (a), (b), (c) for $\cC$ follow by Lemma~\ref{lem.Ctrim}.
%
\end{proof}


\subsection{Proof of Theorem~\ref{thm.smooth-likeGS} part (d)}

Recall that $\cG$ is a trimmable class of graphs, $\cC$ is the class of connected graphs in $\cG$, and for some $k$ the $k$-cycle $C_k$ is free for $\cG$. Our focus now is on $\cGd$ and $\cCd$.  We shall see that `most' 
of $\cGd$ can be written as a composition, and thus we can use Lemma~\ref{lem.bcr} to shown that $\cGd$ is smooth; and similarly for $\cCd$.

%
Let $\Hb_0$ be a $k$-cycle $C_k$, rooted at a vertex $r_0$.
Form $\Hb_1$ from $\Hb_0$ by adding a new vertex $r_1$ and edge $r_0r_1$, and make $r_1$ the root of $\Hb_1$.  Observe that $\Hb_1$ is free for~$\cG$.
Let $\cD$ be the class of graphs in $\cG^{\delta \ge 2}$ with no pendant appearance of $\Hb_1$ and no component $C_k$.
 Given a graph $G \in \cG$ (other than the null graph), the {\em inner core} $\iCore(G)$ is the induced subgraph of $G$ in $\cD$, or possibly the null graph, defined uniquely by the following trimming process and final tidying.  We may also define it as the unique maximal induced subgraph of $G$ in $\cD$, or the null graph if there is no such induced subgraph. 

  We let $i\Core(G)$ be the union over the components $H$ of $G$ of $i\Core(H)$. Consider a connected graph $G \in \cC$.  In the main procedure we start with $G$, and repeatedly trim off pendant copies of $\Hb_1$ and leaves (and delete the attaching bridges), until we can continue no further. 
At each stage the current graph must be in $\cC$.  Let $G'$ be the final graph remaining. 
For example, suppose that $G$ consists of two disjoint $k$-cycles 
joined by a path $v_0,\ldots,v_j$ with $j \geq 0$ edges.
If $j$ is 0 or 1 then $G'=G$ and we are in case~(c) -- see below.
For each $j \geq 2$ we are in case (a) or (b):
if $j$ is 2 or~3 then $G'$ is one of the two $k$-cycles; and if 
$j \geq 4$ then $G'$ is one of the two $k$-cycles or is one of the vertices $v_2,\ldots, v_{j-2}$ in the path.

It is easy to see 
that $G'$ must be (a) a single vertex, or (b) a $k$-cycle,
 or (c) a graph in $\cD$.
The key fact is that, in case (c) when we end up with a graph in $\cD$, every run of the trimming process must end up with this same graph -- see below.
In case (c) we set $i\Core(G)$ to be this graph: in the other cases we set $i\Core(G)$ to be the null graph. Also we let the $i\Core$ of the null graph be the null graph. Thus $\iCore(G)$ is uniquely defined.

Let $\cEr$ be the class of rooted graphs $G \in \cC$ where the trimming process can end up in case (a), with an isolated vertex which we make the root, and where each vertex in $G$ except perhaps the root has degree at least~2.
 Further, let $\cH$ be the subclass of $\cGd$ consisting of the (not necessarily connected) graphs with minimum degree at least 2 such that the trimming process ends up with no components in $\cD$ (so each component is in case (a) or case (b)).  Observe that each 
graph in $\cH$ is free for $\cG$.

Let us establish the uniqueness mentioned above for case (c), and the equivalent definition of the icore as the unique maximal induced subgraph in $\cD$.
Let $G$ be a connected graph.  Call a non-empty set $W$ of vertices in $G$ \emph{safe} if the induced subgraph $G[W]$ has minimum degree at least 2, is not the $k$-cycle $C_k$, 
and has no pendant appearance of $\Hb_1$; that is, if the induced subgraph $G[W]$ is in $\cD$.

\noindent
\emph{Claim} \; A union of safe sets is safe.

Let $W_1$ and $W_2$ be safe sets. We want to show that $W :=W_1 \cup W_2$ is safe. Clearly each vertex degree in $G[W]$ is at least 2. 
If $G[W]$ is $C_k$ and $v \in W_1$ then both neighbours $x$ and $y$ of $v$ in the cycle must be in $W_1$ (since $G[W_1]$ has minimum degree at least 2), and similarly for $x$ and $y$; and it follows easily that $W=W_1$ and $G[W_1]$ is a $k$-cycle, contradicting our choice of $W_1$ as safe.  Thus $G[W]$ cannot be $C_k$.
 %
 %
To prove the claim, it remains only to check that $G[W]$ has no pendant appearance of $\Hb_1$.  But suppose that $G[W]$ has a pendant appearance of $\Hb_1$, on vertex set $Y \subset W$, with the root vertex $r_1$ of $\Hb_1$ adjacent to vertex $v \in W \backslash Y$.
Let $r_1$ be in $W_1$ say.  Then both $v$ and $r_0$ must be in $W_1$ (since $r_1$ has degree 2 in $G[W]$ and degree at least 2 in $G[W_1]$).  Similarly, at least one of the neighbours of $r_0$ on the $k$-cycle $C_k$ (in the copy of $\Hb_1$) must be in $W_1$, and we may carry on round the $k$-cycle to see that all of it is in $W_1$; and thus $G[W_1]$ has a pendant appearance of $\Hb_1$ -  contradicting our choice of $W_1$ as safe. We have now established the above claim.

If there is no safe set in $G$, then the trimming process clearly must end up in case (a) or (b).
Suppose that there is a safe set in $G$.  Let $W^*$ be the union of all the safe sets, and note that $W^*$ is safe by the claim above.
Suppose that the trimming process ends up with a connected graph $H$.  We shall see that $V(H) = W^*$.
We claim first that $V(H) \supseteq W^*$.  For if not, then there must be a vertex set $U \supseteq W^*$  such that in $G[U]$ there is a 
pendant appearance 
of $\Hb_1$ containing a vertex $v \in W^*$.  Now, using the fact that in $G[W^*]$ each vertex has degree at least~2, we can argue as before that we have a pendant appearance of $\Hb_1$ in $G[W^*]$, which contradicts $W^*$ being safe.
This establishes the claim that $V(H) \supseteq W^*$.
But also we cannot have $V(H) \supset W^*$, since then $V(H)$ would be a safe set larger than $W^*$.  Thus $V(H)=W^*$, and $H$ is uniquely specified as the graph  $G[W^*]$.

The above discussion yields the following lemma.  Given a class $\cA$ of graphs we write $\cA^{conn}$ for the class of connected graphs in $\cA$.  (Thus $\cC$ could also be written as $\cG^{conn}$.)

\begin{lemma} \label{lem.struct}
Let $\cG$ be a trimmable class of graphs, let $\cC$ be the class of connected graphs in $\cG$, and suppose that for some $k$ the $k$-cycle $C_k$ is free for $\cG$. Define $\cD$, $\cEr$ and $\cH$ as above. Then
\begin{equation} \label{eqn.Gstruct}
 \cG^{\delta \ge 2} = (\cD \circ \cEr) \, \otimes \cH
\end{equation}
and 
\begin{equation} \label{eqn.Cstruct}
\cC^{\delta \ge 2} = (\cD^{conn} \circ \cEr) \cup \cH^{conn},
\end{equation}
where in both cases the composition is unrooted.
\end{lemma}

We can now 
show that $\cG^{\delta \ge 2}$ is smooth.

\begin{lemma} \label{lem.Gsmooth}
Let the class $\cG$ of graphs be trimmable and have a weak growth constant, and assume that some cycle 
$H_0$
is free in $\cG$.  Then $\rho_{\cG}< e^{-1}$ and $\cG^{\delta \ge 2}$ is smooth.
\end{lemma}

\begin{proof}
The class $\cF$ of forests is contained in $\cG$, and forests have no pendant appearance of $H_0$.
Hence by 
Lemma~\ref{thm.pendapp-nogc} we have $e^{-1}= \rho_{\cF} > \rho_{\cG}$. 
Thus by Lemma~\ref{lem.Gtrim} 
$\,\cGd$ has a weak growth constant.

Next we show that $ \rho_{\cH} > \rho(\cG^{\delta \ge 2})$.
To see this, pick a root vertex in $H_0$ to form $\Hb_0$; and let $\Hb_3$ be the rooted graph formed from three copies of $\Hb_0$ by adding a new vertex $r_3$ adjacent to the root of each copy of $\Hb_0$, and making $r_3$ the root of $\Hb_3$.  Then $\Hb_3$ is
free for $\cGd$, and the subclass $\cH$ of $\cG^{\delta \ge 2}$ has no pendant appearance of $\Hb_3$. Hence $ \rho_{\cH} > \rho(\cG^{\delta \ge 2})$ by
Lemma~\ref{thm.pendapp-nogc}.

By equation~(\ref{eqn.Gstruct}) in Lemma~\ref{lem.struct}
\begin{equation} \label{eqn.geq2}
 \cG^{\delta \ge 2} = (\cD \circ \cEr) \, \otimes \cH \end{equation}
(where the composition is unlabelled).
Hence, $ \rho_{\cH} > \rho(\cG^{\delta \ge 2}) = \rho(\cD \circ \cEr)$ by equation~(\ref{eqn.rhotimes}).  Therefore, by Lemma~\ref{lem.gcstimes}, in order to show that $\cGd$ is smooth, it suffices to show that  $\cD \circ \cEr$ is smooth.
%
%
Observe first that $\cEr \subseteq \cHb$, so
\[ \rho_{\cEr} \geq \rho_{\cHb} = \rho_{\cH} >  \rho(\cD \circ \cEr). \]
We noted above that $\cGd$ has a weak growth constant.  But 
$ \rho_{\cH} > \rho(\cD \circ \cEr)$, 
so by equation~(\ref{eqn.geq2}) and Lemma~\ref{lem.wgctimesconv}, $\cD \circ \cEr$  has a weak growth constant.
Also $\rho_{\cEr} > \rho(\cD \circ \cEr)$, so by Lemma~\ref{lem.Gwgc}, $\cD$ has a weak growth constant. Hence, noting that $\cE$ is aperiodic, we see by Lemma~\ref{lem.bcr} that $\cD \circ \cEr$ is smooth, as required.
\end{proof}


The final lemma in this section shows that $\cC^{\delta \ge 2}$ is smooth: to prove it we argue much as for the last lemma.
\begin{lemma} \label{lem.Csmooth}
  Let $\cC$ be the class of connected graphs in a trimmable class $\cG$, and let $\cC$ have a weak growth constant.  Assume that some 
cycle $H_0$ 
is free for $\cG$. Then $\rho_{\cC}<e^{-1}$ and $\cCd$ is smooth.
\end{lemma}
\begin{proof}
The class $\cT$ of trees is contained in $\cG$, and trees have no pendant appearance of $H_0$.
Hence by 
Lemma~\ref{thm.pendapp-nogc} we have $e^{-1}= \rho_{\cT} > \rho_{\cC}$. By Lemma~\ref{lem.Ctrim}, $\cCd$ has a weak growth constant.
Also, $\rho_{\cH} > \rho(\cC^{\delta \ge 2}) $, by 
considering pendant copies of $\cHb_3$ as in the proof of Lemma~\ref{lem.Gsmooth}.  Thus
\[ 
\rho_{\cEr} \geq \rho_\cH > \rho(\cD^{conn} \circ \cEr).\]
By equation~(\ref{eqn.Cstruct}) in Lemma~\ref{lem.struct}
\[ \cC^{\delta \ge 2} = (\cD^{conn} \circ \cEr) \cup \cH^{conn}. \]
Since $\cCd$ has a weak growth constant and $\rho_\cH > \rho(\cCd)$ it follows that  $\cD^{conn} \circ \cEr$ has a weak growth constant. But $\rho_{\cEr} > \rho(\cD^{conn} \circ \cEr)$, 
 so by Lemma~\ref{lem.wgctimesconv} $\cD^{conn}$ has a weak growth constant.
But now by Lemma~\ref{lem.bcr},  $\cD^{conn} \circ \cEr$ is smooth.  Finally, $\cCd$ is smooth by Lemma~\ref{lem.gcstimes} since $\rho_{\cH} > \rho(\cC^{\delta \ge 2}) $.
\end{proof}

The above two lemmas yield 
part (d) of Theorem~\ref{thm.smooth-likeGS}, which completes the proof of that theorem.


\section{Proof of Theorem~\ref{thm.smooth-Frag}}
\label{sec.proofthm3}


Suppose that $\cG$ and $\cA$ satisfy the conditions in Theorem~\ref{thm.smooth-Frag}.
Let us first check that the conditions in Theorem~\ref{thm.smooth-likeGS} hold.
Since $\cG$ is closed under taking subgraphs, $\cA$ must also have this property.  Thus $K_1$ is free for $\cG$; that is, $\cG$ is trimmable.  Also it is easy to see that $\cA$ is addable, 
as required.  Hence the conditions in Theorem~\ref{thm.smooth-likeGS} do indeed hold.

We shall give 
six preliminary lemmas, in the five subsections~\ref{subsec.1} to~\ref{subsec.5}, and then complete the proof of Theorem~\ref{thm.smooth-Frag} in subsection~\ref{subsec.6}.

\subsection{Convergence of $\Frag(R_n)$ to $BP(\cA,\rho)$}
\label{subsec.1}
 
We recall part of Lemma 7.2 of~\cite{cmcd2013} (for the uniform distribution).

  %

\begin{lemma} \label{lem.stillgen2}
  Let $\cG$ be a bridge-addable smooth class of graphs; let $\rho := \rho_{\cG}$,  and let $\cA$ be the class of graphs which are addable and removable in $\cG$.
Let $R_n \in_u \cG$, and suppose that $\,\Frag(R_n) \in \cA$ whp. 
Then $A(\rho)$ is finite; and 
$\Frag(R_n)$ converges in distribution to the Boltzmann Poisson random graph $BP(\cA, \rho)$.
\end{lemma}
 
\subsection{Fragment and core} 
\label{subsec.2}
 
We shall want to know when the fragment of the core of a graph is the same as the core of the fragment.
\begin{lemma} \label{lem.core-frag}
Let the graph $G$ satisfy $\,\core(G) > 2\, \frag(G)$.  Then \[\Frag(\Core(G))=\Core(\Frag(G))\,.\]
\end{lemma}
\begin{proof}
Since
\[ \core(\Frag(G)) \leq \frag(G) < \tfrac12 \core(G)\]
we have
\[ \core(\Bigc(G)) > \tfrac12 \core(G) > \core(\Frag(G))\,,\]
and the conclusion of the lemma follows.
\end{proof}

\subsection{The core of a Boltzmann Poisson random graph}
\label{subsec.3}

Theorem~\ref{thm.smooth-Frag} refers to the core of the fragment of a random graph $R_n\inu \cG$.  We have often seen that the fragment of $R_n$ converges in distribution to a Boltzmann Poisson random graph.
We show that, under rather general conditions, the core of a Boltzmann Poisson random graph $BP(\cG, \rho)$ has distribution $BP(\cGd,\rho_2)$ (for the natural value of $\rho_2$).
It will follow then that the core of the fragment of $R_n$ must converge in distribution to a Boltzmann Poisson random graph.


\begin{lemma} \label{lem.corefrag}
Let the graph class $\cG$ be decomposable and trimmable, and suppose that $\cGd \neq \emptyset$. 
Let $\rho_0>0$ satisfy $G(\rho_0)<\infty$ (so $\rho_0 \leq 1/e$), and let $R \sim BP(\cG, \rho_0)$.
Let $\rho_2=T^{\bullet}(\rho_0)$, so $\rho_2$ is the unique root $x$ with  $0<x<1$ to $x e^{-x}=\rho_0$.
Then $G^{\delta \geq 2}(\rho_2) < \infty$, and
\[ \Core(R) \sim BP(\cGd, \rho_2).\]
\end{lemma}

\begin{proof}
Let $H$ be a (fixed) connected graph in $\cGd$. Let $\cH$ be the set of connected graphs $G \in \cG$ such that $\Core(G)$ is isomorphic to $H$, and let $\tilde{\cH}$ be the corresponding set of unlabelled graphs.  Let
\[ \mu_G =  \frac{\rho_0^{v(G)}}{\aut(G)} \;\; \mbox{ for } G \in \tilde{\cG}.\]
Denote $\Core(R)$ by $R'$. Then $\kappa(R',H) = \kappa(R,\cH)$, which has distribution $\Po(\lambda)$ where $\lambda =  \sum_{G \in \tilde{\cH}} \mu_G$. Also, for distinct unlabelled connected graphs $H_1,H_2,\cdots$ in $\tilde{\cG}^{\delta \geq 2}$, the corresponding sets $\tilde{\cH}_1, \tilde{\cH}_2,\ldots$ are disjoint, so $\kappa(R', H_1), \kappa(R', H_2),\ldots$ are independent. Thus it remains to show that
\begin{equation} \label{eqn.showsum}
  \sum_{G \in \tilde{\cH}} \mu_G = \frac{\rho_2^{v(H)}}{\aut(H)}.
\end{equation}
(It suffices to consider only the connected graph $H$, see Theorem 1.3 of~\cite{cmcd2009}.)

Let $F_H(x) = \frac{x^{v(H)}}{\aut(H)}$. Then the generating function $H(x)$ for the (labelled) graphs $G$ with $\Core(G)$ isomorphic to $H$ is given by the composition formula $H(x) = F_H(T^{\bullet}(x))$.  Now recall that $T^{\bullet}(x)=x e^{T^{\bullet}(x)}$, so $x= T^{\bullet}(x) e^{-T^{\bullet}(x)}$,
  and thus $\rho_2= T^{\bullet}(\rho_0)$ satisfies $x e^{-x}=\rho_0$, as claimed. 
Also,
\begin{equation} \label{eqn.Gdecomp}
  G(x)= G^{\delta \geq 2}(T^{\bullet}(x)) \cdot F(x)
\end{equation}
(where $F(x)$ is the generating function for the forests).
Thus $G(\rho_0)= G^{\delta \geq 2}(\rho_2) \cdot F(\rho_0)$, and so $G^{\delta \geq 2}(\rho_2) < \infty$.

Finally let us check~(\ref{eqn.showsum}): we find
\begin{eqnarray*}
  \sum_{G \in \tilde{\cH}} \mu_G
    & = &
 \sum_n \sum_{G \in \tilde{\cH}_n} \mu_G  \;\;
    =
 \sum_n \sum_{G \in \tilde{\cH}_n} \frac{\rho_0^n}{\aut(G)}\\
   & = &
 \sum_n  \rho_0^n \sum_{G \in \tilde{\cH}_n} \frac{1}{\aut(G)}
   \;\; = \;\;
 \sum_n  \rho_0^n \frac{| \cH_n|}{n!} \; = H(\rho_0) \\
   & = &
  F_H(T^{\bullet}(\rho_0)) \; = \;\
   F_H(\rho_2) \;\ = \; \frac{\rho_2^{v(H)}}{\aut(H)},
\end{eqnarray*}
  as required.
\end{proof}

\subsection{Sublinearly limited graphs}
\label{subsec.4}

The next two lemmas concern induced subgraphs that cannot often appear disjointly in a graph from a given class $\cG$ of graphs. Given 
a connected graph $H$, for each $n \in \N$ we let $k_n(\cG,H)$ be the maximum number of vertex disjoint copies of $H$ in a graph $G \in \cG_n$. For example, if $\cG$ is the class of graphs embeddable in a surface $S$ of Euler genus $g$, and $H$ is non-planar, then $k_n(\cG,H) \leq g$.  If $k_n(\cG,H) = o(n)$ we say that $H$ is \emph{sublinearly limited} in $\cG$.  Observe that if $\cG$ is closed under taking subgraphs and $H \not\in \cG$ then trivially $k_n(\cG,H)=0$. 
\begin{lemma}  \label{lem.sublin}
Let the class $\cG$ of graphs be bridge-addable and closed under taking subgraphs.
Let $\cA$ be the subclass of graphs in $\cG$ which are 
addable to $\cG$, and suppose that $0< \rho_\cA=\rho_\cG < \infty$. Let $H$ be a connected graph in $\cG \setminus \cA$. Then $H$ is sublinearly limited in $\cG$.
\end{lemma}
Note that in this lemma the graphs which  are addable to $\cG$ are in fact free for $\cG$.
\begin{proof}
Write $k_n$ for $k_n(\cG,H)$, and suppose for a contradiction that
\[ \limsup_{n \to \infty}\, k_n/n = \nu >0. \]
Let $\rho:= \rho_\cA=\rho_\cG$, and let $h=v(H)$.
Let $0 < \eta \leq \nu/2$ be sufficiently small that $2 \eta\, h < 1$ and
\[\left(\frac{\rho^h h}{2\eta\,(\aut H)}\right)^{\eta} = 1+ \delta >1 \]
for some $0<\delta < \frac12$.
The class $\cA$ is bridge-addable and decomposable, so it is addable; and $0< \rho_\cA < \infty$. Hence 
$\cA$ has growth constant $\rho^{-1}$, by Theorem 3.3 of~\cite{msw05}.
Let $j=j(n) = \lceil \eta n \rceil$ for each $n \in \N$.
Let $n_0$ be sufficiently large that for each $n \geq n_0/2$
\begin{equation} \label{eqn.cAn}
  |\cA_n| \geq n! \, \rho^{-n} (1-\delta/2)^n \,,\end{equation}
and for each $n \geq n_0$ we have $n \geq 2hj$ and  
\begin{equation} \label{eqn.cGn}
 |\cG_n| <  n!\, \rho^{-n} (1+\delta/4)^n.
\end{equation}

Let $n \geq n_0$ be such that $k_n \geq j$.
We construct many graphs in $\cG_n$ as follows.  Start with any graph on $[n]$ consisting of $j$ disjoint copies of $H$ together with isolated vertices (this graph must be in $\cG_n$); 
add a graph in $\cA$ on the set $W$ of $n-jh \geq n/2 \geq n_0/2$ vertices not in the copies of $H$; 
and finally form $G \in \cG_n$ by adding, for each copy of $H$, a link edge between a vertex in $H$ and a vertex in $W$.  The number of constructions is at least
\[ \frac{(n)_{jh}}{j! \, (\aut\, H)^j} \, |\cA_{n-jh}| \, (h\, n/2)^j .\]
But $\cA$ is closed under taking subgraphs, so graphs in $\cA$ cannot have a component $H$ or a pendant appearance of $H$; and thus each graph $G$ is constructed just once.  Hence, using the three inequalities~(\ref{eqn.cAn}), $j! \leq (\eta n)^j$ (for $n$ sufficiently large), and $(1+\delta)(1-\delta/2) > 1+\delta/4$, we have
\begin{eqnarray*}
|\cG_n| & \geq &
\frac{(n)_{jh}}{j! \,(\aut\, H)^j} (n-jh)!\, \rho^{-n+jh} (1-\delta/2)^n \, (h \,n/2)^j\\
& \geq &
n! \, \rho^{-n}\, \left(\frac{\rho^h hn}{2 \, \eta n \, (\aut\, H)} \right)^j \,(1-\delta/2)^n\\
& \geq &
n! \, \rho^{-n}\, (1+\delta)^n (1-\delta/2)^n\\
& > &
n! \, \rho^{-n}\, (1+\delta/4)^n.
\end{eqnarray*}
This contradicts~(\ref{eqn.cGn}) and completes the proof.
\end{proof}

The following lemma is related to Lemma 4.6 in~\cite{cmcd2008} and Lemmas 3.3 and 3.4 in~\cite{cmcd2013}.

\begin{lemma} \label{lem.Frag2}
Let the class $\cG$ of graphs be bridge-addable, and let $\cA$ be a decomposable subclass of $\cG$.
Suppose that each connected graph $H$ not in $\cA$ is sublinearly limited in $\cG$.  
Then for $R_n \in_u \cG$, whp $\Frag(R_n) \in \cA$.
\end{lemma}
 

\begin{proof} 
Let $H$ be a connected graph not in $\cA$, and
write $k_n$ for $k_n(\cG,H)$.  Let $\cB$ be the set of graphs $G \in \cG$ such that $\Frag(G)$ has a component $H$ and $\frag(G) \leq n/3$, where $n=v(G)$.
   We shall show first that
\begin{equation} \label{eqn.limited1}
  \pr(R_n \in \cB) \leq \frac{3\,k_n}{2\,v(H)\,n}.
\end{equation}
Given a graph $G \in \cB_n$, add any edge between a component $H$ and a vertex in $\Bigc(G)$, to form $G'$. Since $\bigc(G) \geq 2n/3$ this gives at least $v(H) \cdot 2n/3$ constructions of graphs $G' \in \cG_n$.
%
But each graph $G'$ constructed can have at most $k_n$ pendant copies of $H$; and thus $G'$ can be constructed at most $k_n$ times.  Hence
\[
  v(H) \frac{2n}{3} |\cB_n| \leq k_n |\cG_n|
\]
and~(\ref{eqn.limited1}) follows. \smallskip

Now let $\eps>0$.  Since $\cG$ is bridge-addable, $\E[\frag(R_n)]<2$ (see~\cite{cmcd2013}),   and so $\pr(\frag(R_n) \geq 4/\eps) < \eps/2$ by Markov's inequality.
List the connected graphs not in 
$\cA$ 
in order of non-decreasing number of vertices as $H_1, H_2, \ldots$.  Let $j_0$ be sufficiently large that $v(H_{j}) > 4/\eps$ for any $j > j_0$.  Assume that $n \geq 12/\eps$.  Then, using~(\ref{eqn.limited1}) for the third inequality below (and noting that $4/\eps \leq n/3$),
\begin{eqnarray*}
   && \pr(\Frag(R_n) \not\in \cA) \\
  &=&
    \pr(\Frag(R_n) \mbox{ has a component not in } \cA)\\
  & \leq &
    \pr(\Frag(R_n) \mbox{ has a component in } \{H_1,\ldots,H_{j_0}\})
    +  \pr(\frag(R_n) > 4/\eps)\\
  & \leq &
    \sum_{j=1}^{j_0} \pr(\Frag(R_n) \mbox{ has a component } H_j) 
     + \eps/2 \\
  & \leq &
    \sum_{j=1}^{j_0} \frac{3\,k_n(\cG,H_j)}{2\,v(H_j)\,n} + \eps/2 \;\; = \;\; o(1) + \eps/2 \;\; < \eps
\end{eqnarray*}
  if $n$ is sufficiently large.
\end{proof}

\subsection{Graphs which are addable and removable in $\cG^{\delta \geq 2}$}
\label{subsec.5}

We need one last easy deterministic lemma.
\begin{lemma} \label{lem.fa2}
Let the class $\cG$ of graphs be trimmable with $\cGd$ non-empty, let $\cA$ be the class of graphs which are addable and removable in 
$\cG$, and let $\cB$ be the class of graphs which are addable and removable in 
$\cGd$. Then $\cB = \cAd$.
\end{lemma}
\begin{proof}
Let $H \in \cAd$.  Consider any graph $G$.  If $G \cup H \in \cGd$ then $G \in \cG$ (since $G \cup H \in \cG$ and $H \in \cA$) so $G \in \cGd$.
If $G \in \cGd$ then $G \cup H \in \cG$ (since $H \in \cA$) so $G \cup H \in \cGd$.  Hence $H \in \cB$, and it follows that $\cAd \subseteq \cB$.

Now let $H \in \cB$.   Consider any graph $G$.
If $G \cup H \in \cG$ then $\Core(G) \cup H \in \cGd$, so $\Core(G)$ is in $\cGd$ or is the null graph, and so $G \in \cG$.
If $G \in \cG$ then $\Core(G)$ is in $\cGd$ or is the null graph; so $\Core(G) \cup H \in \cGd$, and thus $G \cup H \in \cG$.
Hence $H \in \cA$, and it follows that $\cB \subseteq \cAd$.
\end{proof}

\subsection{Completing the proof of Theorem~\ref{thm.smooth-Frag}}
\label{subsec.6}

We may now use the assembled lemmas above, Lemma~\ref{lem.stillgen2} to Lemma~\ref{lem.fa2}, to complete the proof of Theorem~\ref{thm.smooth-Frag}.

\begin{proof}[Proof of Theorem~\ref{thm.smooth-Frag}]

Let $\cG$ and $\cA$ satisfy the conditions in Theorem~\ref{thm.smooth-Frag}.
Since $\cG$ is closed under taking subgraphs, $\cA$ must also have this property.  Thus $K_1$ is in $\cA$ and so $K_1$ is free for $\cG$; that is, $\cG$ is trimmable.  Also $\cA$ is addable, and $\cA \supset \cF$.  Hence $\cG$ and $\cA$ satisfy the conditions in Theorem~\ref{thm.smooth-likeGS} (as stated in the theorem).  Let $\rho_2$ be as in Theorem~\ref{thm.smooth-likeGS}.



By Lemma~\ref{lem.sublin}, each connected graph not in $\cA$ is sublinearly limited in~$\cG$. 
Let $R_n \in_u \cG$.  Then, by Lemma~\ref{lem.Frag2}, \whp $\Frag(R_n) \in \cA$.  Hence, by Lemma~\ref{lem.stillgen2}, $A(\rho)$ is finite, and the 
fragment of $R_n$ converges in distribution to $BP(\cA, \rho)$. 


Now consider cores as well as fragments. 
By part (c) of Theorem~\ref{thm.smooth-likeGS}, there exists $\alpha>0$ such that whp $\core(R_n) \geq \alpha n$.  But $\E[\frag(R_n)]<2$, and it follows easily (using Markov's inequality) that whp $\core(R_n) \gg \frag(R_n)$.  Hence by Lemma~\ref{lem.core-frag} whp
$\Frag(\Core(R_n))=\Core(\Frag(R_n))$.
Also, since $\Frag(R_n)$ converges in distribution to $BP(\cG, \rho_0)$ it follows that $\Core(\Frag(R_n))$ converges in distribution to $\Core(BP(\cG, \rho_0))$ and thus by Lemma~\ref{lem.corefrag} to $BP(\cG^{\delta \geq 2}, \rho_2)$.

Finally consider $\Frag(R'_n)$ where $R'_n \inu \cG^{\delta \geq 2}$. 
By Theorem~\ref{thm.smooth-likeGS} part (d), $\cGd$ is smooth.
Observe that $\cGd$ is bridge-addable, and by Lemma~\ref{lem.fa2} $\cAd$ is the class of graphs which are addable and removable in $\cGd$. 
%
By Lemma~\ref{lem.sublin} each connected graph not in $\cA$ is sublinearly limited in $\cG$.
Thus by Lemma~\ref{lem.Frag2} applied to $\cGd$ and $\cAd$, \whp $\Frag(R'_n)$ is in $\cAd$.
Hence, by Lemma~\ref{lem.stillgen2} applied to $\cGd$ and $\cA^{\delta \geq 2}$, we see that
the fragment of $R'_n$ 
converges in distribution to $BP(\cAd, \rho_2)$,
as required. 

We have now completed the proof of Theorem~\ref{thm.smooth-Frag}.  But note that, given that the core of $R_n$ has vertex set $W$, the core is uniformly distributed on the graphs in $\cGd$ on $W$. 
Thus the last part of the theorem we proved, that the fragment of $R'_n \in_u \cGd$ converges in distribution to $BP(\cAd, \rho_2)$, easily implies the earlier result that for $R_n \inu \cG$, the fragment of $\Core(R_n)$ has the same limiting distribution.
\end{proof}


\section{Concluding Remarks}
\label{sec.concl}

We have seen how certain results concerning the fragment and the 2-core of random graphs sampled from an addable minor-closed class or a surface class $\cE^S$ (as in Theorem~\ref{thm.smoothold}) can be extended to random graphs sampled from more general graph classes with some similar structure.  The main new theorems, Theorems~\ref{thm.smooth-likeGS} and~\ref{thm.smooth-Frag}, apply to addable minor-closed classes and surface classes but also to many other graph classes.
Theorem~\ref{thm.smooth-likeGS} concerns a trimmable graph class $\cG$ with $0<\rho_\cG < \infty$ which has a suitably large addable subclass, and Theorem~\ref{thm.smooth-Frag} has slightly stronger assumptions.
We have focussed on showing smoothness of $\cG$ and related graph classes, showing that the limiting distribution of the fragment of a corresponding random graph $R_n$ is a Boltzmann Poisson distribution, determining the order of the core of $R_n$, and considering the fragment of the core.




\end{document}